\documentclass[1p]{elsarticle}

\usepackage{latexsym,amssymb,amsmath,amsthm,amsfonts,amscd}
\usepackage{enumerate}
\usepackage{graphicx}
\usepackage{tikz,tkz-graph,tkz-berge}
\usetikzlibrary{matrix,arrows,calc}
\usepackage{multicol}
\usepackage{multirow}
\usepackage{tabularx}
\usepackage{listings}
\usepackage{color}
\usepackage{hyperref}
\usepackage{kbordermatrix}

\definecolor{deepblue}{rgb}{0,0,0.5}
\definecolor{deepred}{rgb}{0.6,0,0}
\definecolor{deepgreen}{rgb}{0,0.5,0}
\definecolor{cof}{RGB}{220,180,150}
\definecolor{pur}{RGB}{186,146,162}
\definecolor{greeo}{RGB}{91,173,69}
\definecolor{greet}{RGB}{52,111,72}

\newtheorem{theorem}{Theorem}
\newtheorem{lemma}[theorem]{Lemma}

\newtheorem{proposition}[theorem]{Proposition}
\newtheorem{corollary}[theorem]{Corollary}

\theoremstyle{definition}
\newtheorem{definition}[theorem]{Definition}

\newcommand{\diag}{\operatorname{diag}}
\newcommand{\SNF}{\operatorname{SNF}}

\def \rank {\operatorname{rank}}
\DeclareMathOperator{\tr}{trs}

\title{The Smith normal form of Laplacian matrices of simplicial annuli and high dimensional trees}

\author[CA]{Carlos A. Alfaro}
\ead[CA]{alfaromontufar@gmail.com}
\address[CA]{
Banco de M\'exico,
Ciudad de M\'exico, M\'exico.
}
\author[UM]{Jesús Uriel Medrano}
\ead[UM]{jesusurielmedrano@gmail.com}
\address[UM]{
Departamento de Matemáticas, Facultad de Ciencias\\
Universidad Nacional Autónoma de México\\
Mexico City, Mexico.
}
\author[JP]{Juan Pablo Serrano}
\ead[JP]{jpserranop@math.cinvestav.mx}
\address[JP]{
Departamento de Matemáticas\\
Cinvestav\\
Mexico City, Mexico.
}
\author[IT]{Iván Téllez Téllez}
\ead[IT]{ivan.tellez@uaslp.mx}
\address[IT]{
Facultad de Economía,\\
Universidad Autónoma de San Luis Potosí,\\
San Luis Potosí, México.
}
\author[RV]{Ralihe R. Villagr\'an}
\ead[RV]{ralihevillagran@gmail.com}
\address[RV]{
Department of Mathematical Sciences, 
Worcester Polytechnic Institute, 
Worcester, USA.
}

\begin{document}


\begin{abstract}
Inspired by the generalization of the formula of determinant of the distance matrix of trees to $k$-trees, obtained in \cite{MR5026280}, we study the Smith normal form of Laplacian matrices associated with some simplicial complexes.
We find relations between sandpile groups of adjacency graphs and the Smith normal form of Laplacian matrices of simplicial complexes.
We use such relations to calculate the Smith normal form of the highest Laplacian matrix of simplicial annuli and $k$-trees.
We also provide numerical experiments to visualize how good are these algebraic invariants to distinguish $k$-trees.
Finally, we point out that the Graham-Lovász-Pollak matrix, used to compute the determinant of the distance matrix of trees, can be used in the context of Laplacian matrices of trees and block graphs.
\end{abstract}


\maketitle


Let $V=\{v_1,\dots,v_{n}\}$.
The elements of $V$ are called {\it vertices}.
A $k$-{\it simplex} is a set of $k$ vertices and its {\it dimension} is $k-1$.
A {\it simplicial complex} $\Sigma$, or {\it complex} for short, is a collection of nonempty simplices on $V$ such that if $\alpha\in\Sigma$, and $\beta\subset\alpha$, then $\beta\in\Sigma$.

\begin{figure}[h!]
    \centering
    \begin{tabular}{c@{\extracolsep{2cm}}c}
        \begin{tikzpicture}[thick,scale=0.9]
        \coordinate (A1) at (0:3);
        \coordinate (A2) at (60:3);
        \coordinate (A3) at (120:3);
        \coordinate (A4) at (180:3);
        \coordinate (A5) at (240:3);
        \coordinate (A6) at (300:3);
        \coordinate (B1) at (0:1.3);
        \coordinate (B2) at (120:1.3);
        \coordinate (B3) at (240:1.3);

        \draw[fill=cof] (A6) -- (B1) -- (A1) -- (A6);
        \draw[fill=cof] (A2) -- (B1) -- (A1) -- (A2);
        \draw[fill=cof] (A2) -- (B2) -- (A3) -- (A2);
        \draw[fill=cof] (A4) -- (B2) -- (A3) -- (A4);
        \draw[fill=cof] (A4) -- (B3) -- (A5) -- (A4);
        \draw[fill=cof] (A6) -- (B3) -- (A5) -- (A6);
        \draw[fill=cof] (B1) -- (B2) -- (A2) -- (B1);
        \draw[fill=cof] (B2) -- (B3) -- (A4) -- (B2);
        \draw[fill=cof] (B3) -- (B1) -- (A6) -- (B3);

        \end{tikzpicture}
        &
        \begin{tikzpicture}[thick,scale=0.7]
    	\tikzstyle{every node}=[minimum width=0pt, inner sep=2pt, circle]
    	\draw   (0:3) node (v11) [draw,label=0:{\footnotesize $v_1$}] {};
    	\draw  (60:3) node (v12) [draw,label=45:{\footnotesize $v_2$}] {};
    	\draw (120:3) node (v13) [draw,label=135:{\footnotesize $v_3$}] {};
    	\draw (180:3) node (v14) [draw,label=180:{\footnotesize $v_4$}] {};
    	\draw (240:3) node (v15) [draw,label=225:{\footnotesize $v_5$}] {};
        \draw (300:3) node (v16) [draw,label=315:{\footnotesize $v_6$}] {};
    	\draw (v11) -- (v12) -- (v13) -- (v14) -- (v15) -- (v16) -- (v11);
        \draw   (0:1.3) node (v21) [draw,label=180:{\footnotesize $v_7$}] {};
    	\draw (120:1.3) node (v22) [draw,label=275:{\footnotesize $v_8$}] {};
    	\draw (240:1.3) node (v23) [draw,label=85:{\footnotesize $v_9$}] {};
    	\draw (v21) -- (v22) -- (v23) -- (v21);
        \draw (v21) -- (v11);
        \draw (v16) -- (v21) -- (v12);
        \draw (v22) -- (v13);
        \draw (v12) -- (v22) -- (v14);
        \draw (v23) -- (v15);
        \draw (v14) -- (v23) -- (v16);
    	\end{tikzpicture}
    \end{tabular}
    
    \caption{A simplicial annulus and its 1-skeleton.}
    \label{fig:complexandunderlying}
\end{figure}
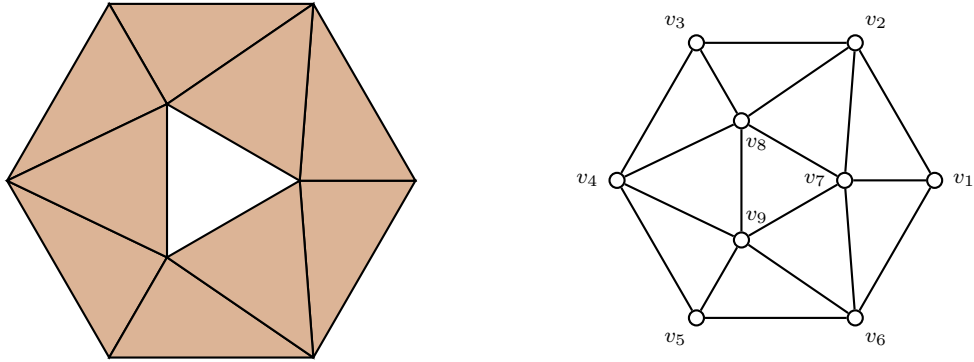

Recall that the {$d$-\it skeleton} of a complex $\Sigma$ is the collection of all simplices of $K$ of dimension at most $d$. 
In particular, the {\it 1-skeleton} of a complex $\Sigma$ is the graph $G(\Sigma)$ defined from the simplices of dimension 0 and 1, see Fig. \ref{fig:complexandunderlying} for an example.
The 1-skeleton of a complex can be computed in {\tt sagemath}\footnote{https://sagecell.sagemath.org} from a {\tt SimplicialComplex} with the {\tt graph()} function.
Consider the following {\tt sagemath} code.
\begin{lstlisting}
O = SimplicialComplex([[6,7,1], [2,7,1], [2,8,3], [4,8,3], [4,9,5], [6,9,5], [7,8,2], [8,9,4], [9,7,6]])
G = O.graph()
G.show()
\end{lstlisting}

Recall that a $d$-{\it clique} $C$ in a graph $G$ is a subset of $d$ vertices of $G$ such that every pair of distinct vertices in $C$ are adjacent.
There is a bijective correspondence between the simplices in $\Sigma$ and a subset $\mathcal{C}$ of the cliques of the 1-skeleton $G(\Sigma)$.
Therefore, a complex $\Sigma$ can be described completely by the 1-skeleton $G(\Sigma)$ by specifying the set $\mathcal{C}$ of {\it allowed cliques} from $\Sigma$.
It is important to note that given the 1-skeleton is not always possible to reconstruct the complex if such allowed cliques are not specified, since some cliques in the 1-skeleton are not necessarily in the complex, consider for example the clique formed by the vertices $v_7$, $v_8$ and $v_9$ in the 1-skeleton in Fig.~\ref{fig:complexandunderlying}.
Therefore, for simplicity, sometimes we will describe a simplicial complex $\Sigma$ by its 1-skeleton $G$ and the set $\mathcal{C}$ of allowed cliques.
For further information on algebraic topology and graph theory we refer the reader to the books \cite{MR256911} and \cite{MR755006}, respectively.

Let $\Sigma$ be a simplicial complex described by the pair $(G,\mathcal{C})$, and let $k$ be the maximum size of a clique in $\mathcal{C}$. 
For $d\in \{1,\dots,k-1\}$, two $d$-cliques $\alpha$ and $\alpha'$ in $\mathcal{C}$ are {\it adjacent} if they belong to the same $(d+1)$-clique $\beta$, in such situation $\alpha$ and $\alpha'$ are {\it incident} to $\beta$. 
The {\it degree} $\deg(\alpha)$ of a $d$-clique $\alpha$ is the number of $d$-cliques adjacent with $\alpha$.
Let $c_d$ denote the number of $d$-cliques in $\mathcal{C}$, and assume that $\alpha_1,\dots,\alpha_{c_d}$ are the $d$-cliques of $\mathcal{C}$.
Then the $d$-{\it adjacency matrix} $A^d(\Sigma)$ of $\Sigma$ is the $c_d\times c_d$ matrix, indexed by the $d$-cliques of $\mathcal{C}$, such that the $(i,j)$-entry is
\[
A^d(\Sigma)_{ij}=
\begin{cases}
    1 & \text{if } \alpha_i \text{ and } \alpha_j \text{ are adjacent, and}\\
    0 & \text{otherwise.}
\end{cases} 
\]
Therefore, the $d$-{\it Laplacian matrix} $L^d(\Sigma)$ of $\Sigma$ is $\diag(\deg(\alpha_1),\dots,\deg(\alpha_{c_d}))-A^d(\Sigma)$.
Sometimes, we will refer to $L^d(\Sigma)$ as the {\it $d$-th Laplacian matrix} of $\Sigma$.
Note that for $d=1$, the 1-adjacency and 1-Laplacian matrices of $\Sigma$ coincide with the adjacency and Laplacian matrices of the 1-skeleton $G$ of $\Sigma$.
The 2-Laplacian matrix of the complex shown in Fig.~\ref{fig:complexandunderlying} is
\[
\tiny
\kbordermatrix{
   & 12 & 16 & 17 & 23 & 27 & 28 & 34 & 38 & 45 & 48 & 49 & 56 & 59 & 67 & 69 & 78 & 79 & 89\\
12 &  2 &  0 & -1 &  0 & -1 &  0 &  0 &  0 &  0 &  0 &  0 &  0 &  0 &  0 &  0 &  0 &  0 &  0 \\
16 &  0 &  2 & -1 &  0 &  0 &  0 &  0 &  0 &  0 &  0 &  0 &  0 &  0 & -1 &  0 &  0 &  0 &  0 \\
17 & -1 & -1 &  4 &  0 & -1 &  0 &  0 &  0 &  0 &  0 &  0 &  0 &  0 & -1 &  0 &  0 &  0 &  0 \\
23 &  0 &  0 &  0 &  2 &  0 & -1 &  0 & -1 &  0 &  0 &  0 &  0 &  0 &  0 &  0 &  0 &  0 &  0 \\
27 & -1 &  0 & -1 &  0 &  4 & -1 &  0 &  0 &  0 &  0 &  0 &  0 &  0 &  0 &  0 & -1 &  0 &  0 \\
28 &  0 &  0 &  0 & -1 & -1 &  4 &  0 & -1 &  0 &  0 &  0 &  0 &  0 &  0 &  0 & -1 &  0 &  0 \\
34 &  0 &  0 &  0 &  0 &  0 &  0 &  2 & -1 &  0 & -1 &  0 &  0 &  0 &  0 &  0 &  0 &  0 &  0 \\
38 &  0 &  0 &  0 & -1 &  0 & -1 & -1 &  4 &  0 & -1 &  0 &  0 &  0 &  0 &  0 &  0 &  0 &  0 \\
45 &  0 &  0 &  0 &  0 &  0 &  0 &  0 &  0 &  2 &  0 & -1 &  0 & -1 &  0 &  0 &  0 &  0 &  0 \\
48 &  0 &  0 &  0 &  0 &  0 &  0 & -1 & -1 &  0 &  4 & -1 &  0 &  0 &  0 &  0 &  0 &  0 & -1 \\
49 &  0 &  0 &  0 &  0 &  0 &  0 &  0 &  0 & -1 & -1 &  4 &  0 & -1 &  0 &  0 &  0 &  0 & -1 \\
56 &  0 &  0 &  0 &  0 &  0 &  0 &  0 &  0 &  0 &  0 &  0 &  2 & -1 &  0 & -1 &  0 &  0 &  0 \\
59 &  0 &  0 &  0 &  0 &  0 &  0 &  0 &  0 & -1 &  0 & -1 & -1 &  4 &  0 & -1 &  0 &  0 &  0 \\
67 &  0 & -1 & -1 &  0 &  0 &  0 &  0 &  0 &  0 &  0 &  0 &  0 &  0 &  4 & -1 &  0 & -1 &  0 \\
69 &  0 &  0 &  0 &  0 &  0 &  0 &  0 &  0 &  0 &  0 &  0 & -1 & -1 & -1 &  4 &  0 & -1 &  0 \\
78 &  0 &  0 &  0 &  0 & -1 & -1 &  0 &  0 &  0 &  0 &  0 &  0 &  0 &  0 &  0 &  2 &  0 &  0 \\
79 &  0 &  0 &  0 &  0 &  0 &  0 &  0 &  0 &  0 &  0 &  0 &  0 &  0 & -1 & -1 &  0 &  2 &  0 \\
89 &  0 &  0 &  0 &  0 &  0 &  0 &  0 &  0 &  0 & -1 & -1 &  0 &  0 &  0 &  0 &  0 &  0 &  2 \\
         }.
\]

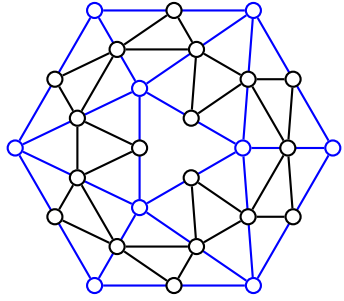
\begin{figure}[!h]
    \centering
        \begin{tikzpicture}[thick,scale=0.7]
    	\tikzstyle{every node}=[minimum width=0pt, inner sep=2pt, circle]
    	\draw   (0:3) node[draw,blue] (v11) {};
    	\draw  (60:3) node[draw,blue] (v12) {};
    	\draw (120:3) node[draw,blue] (v13) {};
    	\draw (180:3) node[draw,blue] (v14) {};
    	\draw (240:3) node[draw,blue] (v15) {};
        \draw (300:3) node[draw,blue] (v16) {};
        \draw   (0:1.3) node[draw,blue] (v21) {};
    	\draw (120:1.3) node[draw,blue] (v22) {};
    	\draw (240:1.3) node[draw,blue] (v23) {};
        \draw[draw,blue] (v11) -- (v12) -- (v13) -- (v14) -- (v15) -- (v16) -- (v11);
        \draw[draw,blue] (v21) -- (v22) -- (v23) -- (v21);
        \draw[draw,blue] (v21) -- (v11);
        \draw[draw,blue] (v16) -- (v21) -- (v12);
        \draw[draw,blue] (v22) -- (v13);
        \draw[draw,blue] (v12) -- (v22) -- (v14);
        \draw[draw,blue] (v23) -- (v15);
        \draw[draw,blue] (v14) -- (v23) -- (v16);
        \node[draw,fill=white] (v31) at ($(v21)!0.5!(v22)$) {};
    	\node[draw,fill=white] (v32) at ($(v22)!0.5!(v23)$) {};
        \node[draw,fill=white] (v33) at ($(v23)!0.5!(v21)$) {};
        \node[draw,fill=white] (v34) at ($(v21)!0.5!(v11)$) {};
        \node[draw,fill=white] (v35) at ($(v16)!0.5!(v21)$) {};
        \node[draw,fill=white] (v36) at ($(v21)!0.5!(v12)$) {};
        \node[draw,fill=white] (v37) at ($(v22)!0.5!(v13)$) {};
        \node[draw,fill=white] (v38) at ($(v12)!0.5!(v22)$) {};
        \node[draw,fill=white] (v39) at ($(v22)!0.5!(v14)$) {};
        \node[draw,fill=white] (v310) at ($(v23)!0.5!(v15)$) {};
        \node[draw,fill=white] (v311) at ($(v14)!0.5!(v23)$) {};
        \node[draw,fill=white] (v312) at ($(v23)!0.5!(v16)$) {};
        \node[draw,fill=white] (v313) at ($(v11)!0.5!(v12)$) {};
        \node[draw,fill=white] (v314) at ($(v12)!0.5!(v13)$) {};
        \node[draw,fill=white] (v315) at ($(v13)!0.5!(v14)$) {};
        \node[draw,fill=white] (v316) at ($(v14)!0.5!(v15)$) {};
        \node[draw,fill=white] (v317) at ($(v15)!0.5!(v16)$) {};
        \node[draw,fill=white] (v318) at ($(v16)!0.5!(v11)$) {};
        \draw (v31) -- (v36) -- (v34) -- (v35) -- (v33) -- (v312) -- (v310) -- (v311) -- (v32) -- (v39) -- (v37) -- (v38) -- (v31);
        \draw (v313) -- (v34) -- (v318) -- (v35) -- (v312) -- (v317) -- (v310) -- (v316) -- (v311) -- (v39) -- (v315) -- (v37) -- (v314) -- (v38) -- (v36) -- (v313);
    	\end{tikzpicture}
    
    \caption{The first (blue) and second (black) adjacency graphs of the complex of Fig.~\ref{fig:complexandunderlying}.}
    \label{fig:auxiliarygraph}
\end{figure}

There have been explored several Laplacian matrices associated with complexes \cite{MR5088845,MR4925290,MR3077874,MR4131346,MR4783080}.
Here we are interested in this variant which does not depend on the orientation of the simplices.

Let $\Sigma$ be a simplicial complex described by the pair $(G,\mathcal{C})$, and let $k$ be the maximum size of any clique in $\mathcal{C}$.
\begin{definition}
    For $d\in \{1,\dots,k\}$, the $d$-{\it th adjacency graph} $\Gamma_d(\Sigma)$ of $\Sigma$ is the graph whose vertex set is the set of $d$-cliques of $\mathcal{C}$, and two vertices in $\Gamma_d(\Sigma)$ are adjacent if and only if their corresponding $d$-cliques are adjacent in $\Sigma$.
\end{definition}
It is important to know that there have been studied many adjacency graphs of complexes \cite{MR4742211,MR427152,fallat2024minimumnumberdistincteigenvalues}, most of them generalizing the concept of line graph, that is, two $d$-cliques are adjacent if and only if their intersection is not empty.
The first and second adjacency graphs of the complex of Fig.~\ref{fig:complexandunderlying} are shown in Fig.~\ref{fig:auxiliarygraph}.
In particular, the 1-skeleton $G$ of $\Sigma$ coincides with $\Gamma_1(\Sigma)$.
In \ref{app:lapmatauxgra}, the reader can found the {\tt d\_laplacian\_matrix} function that computes the $d$-Laplacian matrix of a simplicial complex and its $d$-{\it th} adjacency graph.
For instance, the following line of code computes the second Laplacian matrix of the simplicial complex defined previously.
\begin{lstlisting}
d_laplacian_matrix(O,2)
\end{lstlisting}

The following results follow from the definitions.

\begin{lemma}\label{lem:LmatofAuxandComplexareequal}
    Let $\Sigma$ be a simplicial complex described by the pair $(G,\mathcal{C})$, and let $k$ be the maximum size of a clique in $\mathcal{C}$.
    For $1\leq d\leq k$, let $\Gamma_d$ be the $d$-adjacency graph of $\Sigma$.
    Then, $L^d(\Sigma)=L(\Gamma_d)$ and $A^d(\Sigma)=A(\Gamma_d)$.
\end{lemma}

It is interesting to note that there are non-isomorphic $k$-trees with the same number of vertices such that their $d$-adjacency graphs are isomorphic.
An example of such phenomenon can be observed on the third adjacency graphs of the two 3-trees shown in Figure \ref{fig:placeholder222}.

\begin{figure}[h!]
    \centering
    \begin{tabular}{c@{\extracolsep{2cm}}c}
    \begin{tikzpicture}[scale=0.85,thick]
		\tikzstyle{every node}=[minimum width=0pt, inner sep=2pt, circle]
			\draw (-1.6,2.6) node[draw] (0) { \tiny 0};
			\draw (-1.2,-0.0) node[draw] (1) { \tiny 1};
			\draw (-2.0,0.6) node[draw] (2) { \tiny 2};
			\draw (0.2,-0.6) node[draw] (3) { \tiny 3};
			\draw (-1.8,1.5) node[draw] (4) { \tiny 4};
			\draw (-3.6,-0.6) node[draw] (5) { \tiny 5};
			\draw (-1.2,0.6) node[draw] (6) { \tiny 6};
			\draw  (1) edge (3);
			\draw  (0) edge (3);
			\draw  (0) edge (4);
			\draw  (4) edge (5);
			\draw  (0) edge (5);
			\draw  (2) edge (6);
			\draw  (4) edge (6);
			\draw  (0) edge (6);
			\draw  (1) edge (6);
			\draw  (3) edge (6);
			\draw  (3) edge (5);
			\draw  (2) edge (4);
			\draw  (2) edge (5);
			\draw  (5) edge (6);
			\draw  (1) edge (5);
		\end{tikzpicture}
         &  
    \begin{tikzpicture}[scale=0.9,thick]
		\tikzstyle{every node}=[minimum width=0pt, inner sep=2pt, circle]
			\draw (90:2) node[draw] (0) { \tiny 0};
			\draw (210:1) node[draw] (1) { \tiny 1};
			\draw (330:2) node[draw] (2) { \tiny 2};
			\draw (0.0,0.0) node[draw] (3) { \tiny 3};
			\draw (210:2) node[draw] (4) { \tiny 4};
			\draw (90:1) node[draw] (5) { \tiny 5};
			\draw (330:1) node[draw] (6) { \tiny 6};
			\draw  (2) edge (4);
			\draw  (0) edge (2);
			\draw  (0) edge (4);
			\draw  (1) edge (5);
			\draw  (1) edge (3);
			\draw  (1) edge (6);
			\draw  (1) edge (4);
			\draw  (4) edge (6);
			\draw  (2) edge (6);
			\draw  (0) edge (6);
			\draw  (0) edge (5);
			\draw  (3) edge (6);
			\draw  (3) edge (5);
			\draw  (5) edge (6);
            \draw  (4) edge (5);
		\end{tikzpicture}    
    \end{tabular}
    \caption{Two non-isomorphic 3-trees whose third adjacency graphs are isomorphic.}
    \label{fig:placeholder222}
\end{figure}
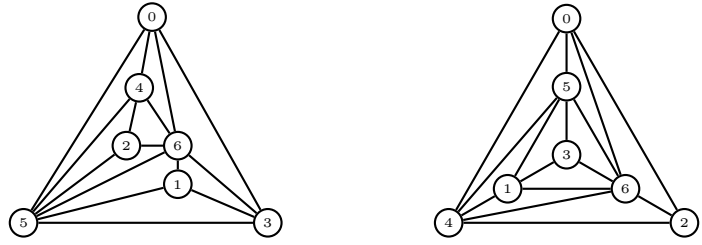

Now let us recall the concept of Smith normal form of a integer matrix.
Two integer matrices $M$ and $N$ are \emph{equivalent} if there exist unimodular matrices $P$ and $Q$ with entries in $\mathbb{Z}$ satisfying $M=PNQ$.
In such case, this relation is denoted by $N\sim M$.
Therefore, $M$ and $N$ are equivalent if and only if $M$ can be transformed into $N$ by means of the following operations:
\begin{enumerate}
  \item swap any two rows or any two columns.
  \item add an integer multiple of one row to another row.
  \item add an integer multiple of one column to another column.
  \item multiply any row or column by $\pm 1$.
\end{enumerate}
The \emph{Smith normal form} of an integer matrix $M$, denoted by $\SNF(M)$, is the unique diagonal matrix $\diag(f_1,\dots,f_r,0,\dots,0)$ equivalent to $M$ such that $r=rank(M)$ and $f_i|f_j$ for $i<j$.
The \emph{invariant factors} (or \emph{elementary divisors}) of $M$ are the integers in the diagonal of the $\SNF(M)$.

The SNF of matrices over principal ideal domains such as $\mathbb{Z}$ and $\mathbb{Q}[x]$ have many applications in algebraic group theory, combinatorics, homology groups, integer programming, lattices, linear Diophantine equations, system theory, and analysis of cryptosystems \cite{cohen,kannan,MR755006,MR957919,schrijver,stanley}.

By considering an $m\times n$ matrix $M$ with integer entries as a linear map $M:\mathbb{Z}^n\rightarrow \mathbb{Z}^m$, the {\it cokernel} of $M$ is the quotient module $\mathbb{Z}^{m}/{\rm Im}\, M$.
It is known that if $N\sim M$, then $coker(M)=\mathbb{Z}^n/{\rm Im} M\cong\mathbb{Z}^n/{\rm Im} N=coker(N)$.
Therefore, as the fundamental theorem of finitely generated Abelian groups states, the cokernel of $M$ can be described as:
$coker(M)\cong \mathbb{Z}_{f_1} \oplus \mathbb{Z}_{f_2} \oplus \cdots \oplus\mathbb{Z}_{f_{r}} \oplus \mathbb{Z}^{n-r}$,
where $r=\rank(M)$, and $f_1, f_2, \dots, f_{r}$ are the {\it invariant factors} of $M$.

Recently, there is much interest in the the finitely generated Abelian groups obtained from the cokernel of a matrix associated with the combinatorial properties of the graph.
Let $G$ be a graph with $n$ vertices.
Let $A(G)$ denote the adjacency matrix of $G$ and let $\deg(G)$ denote the diagonal matrix with the degrees of the vertices of $G$ in the diagonal.
The \emph{Laplacian matrix} $L(G)$ is defined as the matrix $\deg(G)-A(G)$.
The cokernel of the adjacency matrix $A(G)$ is known as the {\it Smith group} of $G$ and is denoted $S(G)$, and the torsion part of the cokernel of the Laplacian matrix $L(G)$ is known as the {\it critical group} or {\it sandpile group} of $G$, denoted by $K(G)$.
For instance, it is not difficult to see that the Laplacian matrix $L(K_n)$ of the complete graph $K_n$ with $n$ vertices is $n{\sf I}_n-{\sf J}_n$ and its Smith normal form is $\diag(1,n,\dots,n,0)$.
From which follows that $K(K_n)=\mathbb{Z}_n^{n-2}$.
Additionally, it is known \cite{klivans} the number of elements in the sandpile group of a graph $G$ coincide with the number $\tau(G)$ of spanning trees of $G$.
Thus, $\tau(K_n)=n^{n-2}$.
And the number of zeros in the diagonal of the Smith normal form of the Laplacian matrix of a graph $G$ counts the number of connected components of $G$.
From which also follows that $\SNF(T)=\diag(1,\dots,1,0)$ for any tree $T$.

Let $\Sigma$ be a simplicial complex described by the pair $(G,\mathcal{C})$, and let $k$ be the maximum size of a clique in $\mathcal{C}$. 
Let $K^d(\Sigma)$ be the torsion part of the cokernel of $L^d(\Sigma)$.
For $1\leq d\leq k$, let $\Gamma_d$ be the $d$-{\it th} adjacency graph of $\Sigma$.
By Lemma~\ref{lem:LmatofAuxandComplexareequal}, $|K^d(\Sigma)|=\tau(\Gamma_d)$.
The following two lemmas are classic in the sandpile group theory.

\begin{lemma}\cite{MR1756151,vince}
    The sandpile group of a connected plane graph is isomorphic to the sandpile group of its dual graph.
\end{lemma}

Therefore, the sandpile group of a connected planar graph is isomorphic to the sandpile group of any of its dual graphs.
Another interesting result is the following.

\begin{lemma}\cite{klivans}\label{lema:bloques}
Let $G$ be a graph with $b$ non-trivial blocks $B_1,\dots,B_b$.
Then $K(G)\simeq K(B_1)\oplus\cdots\oplus K(B_b)$.
\end{lemma}

In particular, previous lemma can be stated in terms of the adjacency graphs of a simplicial complex, which will be used below.

\begin{lemma}
    Let $\Sigma$ be a simplicial complex described by the pair $(G,\mathcal{C})$, and let $k$ be the maximum size of a clique in $\mathcal{C}$.
    For $1\leq d\leq k$, let $\Gamma_d$ be the $d$-adjacency graph of $\Sigma$.
    If $B_1,\dots,B_b$ are the non-trivial blocks of $\Gamma_d$, then the torsion part of cokernel of $L^d(\Sigma)$ coincides with $K(\Gamma_d)$ and $K(B_1)\oplus\cdots\oplus K(B_b)$.
\end{lemma}

These relations between the SNFs of the Laplacian matrices of the simplicial complexes and the sandpile groups of the associated adjacency graphs will be used in the following sections.

The manustript is organized as follows.
In Section~\ref{sec:annuli}, we compute the SNF of the second Laplacian matrix of the simplicial annulus.
In Section~\ref{sec:ktrees}, we compute the SNF of the $k$-{\it th} Laplacian matrix of the $k$-trees.
We also explore numerically how good are the SNFs and the spectrum of all Laplacian matrices to distinguish $k$-trees.
Finally, in Section~\ref{sec:GLP}, we show applications of the Graham-Lovász-Pollak matrix to reduce the Laplacian matrix to a equivalent and, some times, simpler matrix.

\section{Simplicial annuli}\label{sec:annuli}

\begin{figure}[h!]
    \centering
    \begin{tikzpicture}[thick,scale=0.7]
    	\tikzstyle{every node}=[minimum width=0pt,circle]
        \coordinate (A1) at  (54:3);
        \coordinate (A2) at  (90:3);
        \coordinate (A3) at (126:3);
        \coordinate (A4) at  (54:5);
        \coordinate (A5) at  (90:5);
        \coordinate (A6) at (126:5);
        \draw[fill=cof] (A1) -- (A2) -- (A3) -- (A6) -- (A5) -- (A4) -- (A1);
        \draw (A4) -- (A2) -- (A6);
        \draw (A2) -- (A5);
    	\draw (A1) node[inner sep=2pt,draw,fill=white,label=270:{\footnotesize $v_{i-1}$}] (v11) {};
    	\draw (A2) node[inner sep=2pt,draw,fill=white,label=270:{\footnotesize $v_i$}] (v12) {};
    	\draw (A3) node[inner sep=2pt,draw,fill=white,label=270:{\footnotesize $v_{i+1}$}] (v13) {};
    	\draw (A4) node[inner sep=2pt,draw,fill=white,label=90:{\footnotesize $u_{2i-1}$}] (v21) {};
    	\draw (A5) node[inner sep=2pt,draw,fill=white,label=90:{\footnotesize $u_{2i}$}] (v22) {};
        \draw (A6) node[inner sep=2pt,draw,fill=white,label=90:{\footnotesize $u_{2i+1}$}] (v23) {};
        \draw (45:3) node[inner sep=0,draw,fill] {};
        \draw (47:3) node[inner sep=0,draw,fill] {};
        \draw (49:3) node[inner sep=0,draw,fill] {};
        \draw (45:4) node[inner sep=0,draw,fill] {};
        \draw (47:4) node[inner sep=0,draw,fill] {};
        \draw (49:4) node[inner sep=0,draw,fill] {};
        \draw (45:5) node[inner sep=0,draw,fill] {};
        \draw (47:5) node[inner sep=0,draw,fill] {};
        \draw (49:5) node[inner sep=0,draw,fill] {};
        \draw (131:3) node[inner sep=0,draw,fill] {};
        \draw (133:3) node[inner sep=0,draw,fill] {};
        \draw (135:3) node[inner sep=0,draw,fill] {};
        \draw (131:4) node[inner sep=0,draw,fill] {};
        \draw (133:4) node[inner sep=0,draw,fill] {};
        \draw (135:4) node[inner sep=0,draw,fill] {};
        \draw (131:5) node[inner sep=0,draw,fill] {};
        \draw (133:5) node[inner sep=0,draw,fill] {};
        \draw (135:5) node[inner sep=0,draw,fill] {};
    \end{tikzpicture}
    \caption{Local adjacencies in the annulus.}
    \label{fig:annulus}
\end{figure}
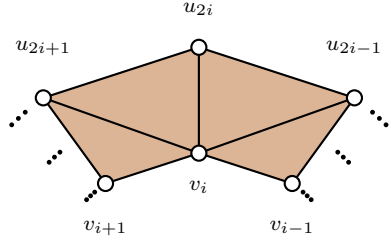

Recall that a {\it simplicial annulus} is a triangulation of the standard topological annulus $S^1\times[0,1]$ using triangles that creates a ring-like structure.
We are going to be particularly interested in the following family of simplicial annuli.
For $n\geq 3$, let $O_n$ consist of two disjoint cycles $C_n$ and $C_{2n}$ whose vertex-sets are $\{v_0,\dots,v_{n-1}\}$ and $\{u_0,\dots,u_{2n-1}\}$, respectively.
For $i\in\{0,\dots,n-1\}$, vertex $v_i$ is adjacent with vertices $u_{2i-1}$, $u_{2i}$ and $u_{2i+1}$, where the sub-indices of the vertices in $C_{2n}$ are taken modulo $2n$.
And the allowed 3-cliques are those of the form $\{v_i,u_{2i-1},u_{2i}\}$, $\{v_i,u_{2i},u_{2i+1}\}$ and $\{v_i,v_{i+1},u_{2i+1}\}$, see Figure~\ref{fig:annulus}.
In \ref{sec:annulussage}, the reader can found a {\tt sagemath} function that returns the simplicial annulus $O_n$.
In this section, we will compute the SNF of the second Laplacian matrix of the simplicial annulus $O_n$.


Let $\widehat{C_n}$ be the graph constructed from a cycle $C_n$ with vertex set $\{v_0,\dots,v_{n-1}\}$ and an empty graph $\overline{K_n}$ with vertex set $\{u_0,\dots,u_{n-1}\}$ such that $u_i$ is adjacent with $v_i$ and $v_{i+1}$, see Figure~\ref{fig:geargraph}.




\begin{figure}[h!]
    \centering
    \begin{tabular}{c@{\extracolsep{1cm}}ccc}
        \begin{tikzpicture}[thick,scale=0.7]
        	\tikzstyle{every node}=[minimum width=0pt, inner sep=2pt, circle]
            \draw   (0:1) node[draw] (v11) {};
        	\draw (120:1) node[draw] (v12) {};
        	\draw (240:1) node[draw] (v13) {};
            \draw  (60:1.5) node[draw] (v21) {};
        	\draw (180:1.5) node[draw] (v22) {};
        	\draw (300:1.5) node[draw] (v23) {};
            \draw (v13) -- (v11) -- (v12) -- (v13);
            \draw (v13) -- (v23) -- (v11) -- (v21) -- (v12) -- (v22) -- (v13);
        \end{tikzpicture}
        &
        \begin{tikzpicture}[thick,scale=0.7]
        	\tikzstyle{every node}=[minimum width=0pt, inner sep=2pt, circle]
            \draw   (0:1) node[draw] (v11) {};
        	\draw  (90:1) node[draw] (v12) {};
        	\draw (180:1) node[draw] (v13) {};
            \draw (270:1) node[draw] (v14) {};
            \draw  (45:1.5) node[draw] (v21) {};
        	\draw (135:1.5) node[draw] (v22) {};
        	\draw (225:1.5) node[draw] (v23) {};
            \draw (315:1.5) node[draw] (v24) {};
            \draw (v14) -- (v11) -- (v12) -- (v13) -- (v14);
            \draw (v14) -- (v24) -- (v11) -- (v21) -- (v12) -- (v22) -- (v13) -- (v23) -- (v14);
        \end{tikzpicture}
        &
        \begin{tikzpicture}[thick,scale=0.7]
        	\tikzstyle{every node}=[minimum width=0pt, inner sep=2pt, circle]
            \draw   (0:1) node[draw] (v11) {};
        	\draw  (72:1) node[draw] (v12) {};
        	\draw (144:1) node[draw] (v13) {};
            \draw (216:1) node[draw] (v14) {};
            \draw (288:1) node[draw] (v15) {};
            \draw  (36:1.5) node[draw] (v21) {};
        	\draw (108:1.5) node[draw] (v22) {};
        	\draw (180:1.5) node[draw] (v23) {};
            \draw (252:1.5) node[draw] (v24) {};
            \draw (324:1.5) node[draw] (v25) {};
            \draw (v15) -- (v11) -- (v12) -- (v13) -- (v14) -- (v15);
            \draw (v15) -- (v25) -- (v11) -- (v21) -- (v12) -- (v22) -- (v13) -- (v23) -- (v14) -- (v24) -- (v15);
        \end{tikzpicture}
        &
        \begin{tikzpicture}[thick,scale=0.7]
        	\tikzstyle{every node}=[minimum width=0pt, inner sep=2pt, circle]
            \draw   (0:1) node[draw] (v11) {};
        	\draw  (60:1) node[draw] (v12) {};
        	\draw (120:1) node[draw] (v13) {};
            \draw (180:1) node[draw] (v14) {};
            \draw (240:1) node[draw] (v15) {};
            \draw (300:1) node[draw] (v16) {};
            \draw  (30:1.5) node[draw] (v21) {};
        	\draw  (90:1.5) node[draw] (v22) {};
        	\draw (150:1.5) node[draw] (v23) {};
            \draw (210:1.5) node[draw] (v24) {};
            \draw (270:1.5) node[draw] (v25) {};
            \draw (330:1.5) node[draw] (v26) {};
            \draw (v16) -- (v11) -- (v12) -- (v13) -- (v14) -- (v15) -- (v16);
            \draw (v16) -- (v26) -- (v11) -- (v21) -- (v12) -- (v22) -- (v13) -- (v23) -- (v14) -- (v24) -- (v15) -- (v25) -- (v16);
        \end{tikzpicture}
    \end{tabular}
    \caption{The graphs $\widehat{C_3}$, $\widehat{C_4}$, $\widehat{C_5}$ and $\widehat{C_6}$.}
    \label{fig:geargraph}
\end{figure}
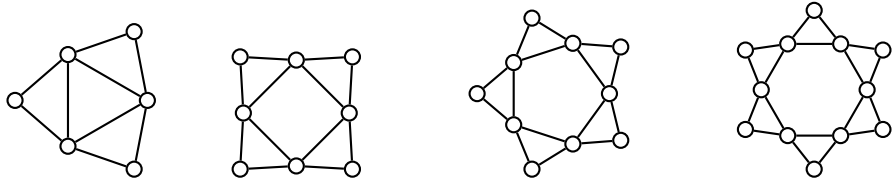

\begin{lemma}\label{lem:secondauxiliarygraphOnisomorphicC3n}
    The second adjacency graph $\Gamma_2(O_n)$ is isomorphic to $\widehat{C_{3n}}$.
\end{lemma}
\begin{proof}
    Let us assume that the vertex set of $O_n$ is $\{v_0,\dots,v_{n-1},v'_0,\dots,v'_{2n-1}\}$ such that $\{v_0,\dots,v_{n-1}\}$ induce the cycle $C_n$ and $\{v'_0,\dots,v'_{2n-1}\}$ induce the cycle $C_{2n}$.
    The vertex set of $\Gamma_2(O_n)$ can be partitioned into three subsets: $V_1 = \{v_iv_{i+1}: 0 \leq i \leq n-1\}$, $V_2 = \{v_i'v_{i+1}': 0 \leq i \leq 2n-1\}$ and $V_3 = \{v_iv_{2i+j}': 0 \leq i \leq n-1,\; j=-1,0,1\}.$
    
    On other hand, let ut assume that the vertex set of $\widehat{C_{3n}}$ is $\{u_0,\dots,u_{3n-1},u'_0,\dots,u'_{3n-1}\}$ such that the vertices $\{u_0,\dots,u_{3n-1}\}$ induce the $C_{3n}$ and the vertices $\{u'_0,\dots,u'_{3n-1}\}$ induce the $\overline{K_{3n}}$.

    We define $\varphi \colon V(\Gamma_2(O_n)) \to V(\widehat{C_{3n}})$ as follows:
    \begin{itemize}
        \item[1] If $v_iv_{i+1} \in V_1$, then $\varphi(v_iv_{i+1}) = u_{3i+1}'$.
        \item[2] If $v_i'v_{i+1}' \in V_2$, then 
        \[
        \varphi(v_i'v_{i+1}') =  \begin{cases}
            u_{\frac{3i}{2}}' & \text{if $i$ is even},\\
            u_{\frac{3i+1}{2}} & \text{otherwise}.
        \end{cases}
        \]
        \item[3] If $v_iv_{2i+j}' \in V_3$, then 
        \[
        \varphi(v_iv_{2i+j}') =  \begin{cases}
            u_{3n-1} & \text{if $i = 0$ and $j = -1$},\\
            u_{3i + j} & \text{otherwise}.
        \end{cases}
        \]
    \end{itemize}

It is easy to see that $\varphi$ is injective on each $V_i$, and hence $\varphi$ is injective. Moreover, since $|V(\Gamma_2(O_n))| = |V(\widehat{C_{3n}})|$, it follows that $\varphi$ is a bijection.

We now show that $\varphi$ is a graph homomorphism.

\textbf{Case 1.} Let $e_1 = v_iv_{i+1} \in V_1$. Note that $e_1$ has degree $2$. Let $e_2$ and $e_3$ be the neighbors of $e_1$. If $e_1 = v_{2n-1}v_{0}$, then $e_2 = v_0v_{2n-1}'$ and $e_3 = v_{n-1}v_{2n-1}'$. Then $\varphi(e_1)\varphi(e_2) = u_{3n-2}'u_{3n-1} \in E(\widehat{C_{3n}})$ and $\varphi(e_1)\varphi(e_3) = u_{3n-2}'u_{3n-2} \in E(\widehat{C_{3n}})$.
If $e_1 \neq v_{n-1}v_0$, then $e_2 = v_iv_{2i+1}'$ and $e_3 = v_{i+1}v_{2i+1}'$. Hence, $\varphi(e_1)\varphi(e_2) = u_{3i+1}'u_{3i+1} \in E(\widehat{C_{3n}})$ and $\varphi(e_1)\varphi(e_3) = u_{3i+1}'u_{3i+2} \in E(\widehat{C_{3n}}).$

\textbf{Case 2.} Let $e_1 = v_i'v_{i+1}' \in V_2$. Note that $e_1$ has degree $2$. Let $e_2$ and $e_3$ be its neighbors. If $e_1 = v_{2n-1}'v_{0}'$, then $e_2 = v_0v_0'$ and $e_3 = v_0v_{2n-1}'$. Then 
$\varphi(e_1)\varphi(e_2) = u_{3n-1}'u_0 \in E(\widehat{C_{3n}})$ and $\varphi(e_1)\varphi(e_3) = u_{3n-1}'u_{3n-1} \in E(\widehat{C_{3n}})$.
If $e_1 \neq v_{n-1}v_0$ and $i$ is even, then $e_2 = v_{\frac{i}{2}}v_i'$ and $e_3 = v_{\frac{i}{2}}v_{i+1}'$. Hence, $\varphi(e_1)\varphi(e_2) = u_{\frac{3i}{2}}'u_{\frac{3i}{2}} \in E(\widehat{C_{3n}})$ and $\varphi(e_1)\varphi(e_3) = u_{\frac{3i}{2}}'u_{\frac{3i}{2}+1} \in E(\widehat{C_{3n}})$. If $e_1 \neq v_{n-1}v_0$ and $i$ is odd, then $e_2 = v_{\frac{i+1}{2}}v_{2i}'$ and $e_3 = v_{\frac{i+1}{2}}v_{i+1}'$. Thus, $\varphi(e_1)\varphi(e_2) = u_{\frac{3i+1}{2}}'u_{\frac{3i+1}{2}}$ and $\varphi(e_1)\varphi(e_3) = u_{\frac{3i+1}{2}}'u_{\frac{3i+1}{2}+1} \in E(\widehat{C_{3n}}).$

\textbf{Case 3.} Let $e_1 = v_iv_{2i+j}' \in V_3$. Note that $e_1$ has degree $4$. Two of its neighbors, say $e_2$ and $e_3$, lie in $V_3$. If $e_1 = v_{0}v_{2n-1}'$, then $e_2 = v_{n-1}v_{2n-1}'$ and $e_3 = v_0v_0'$. Then $\varphi(e_1)\varphi(e_2) = u_{3n-1}u_{3n-2} \in E(\widehat{C_{3n}})$ and $\varphi(e_1)\varphi(e_3) = u_{3n-1}u_{0} \in E(\widehat{C_{3n}})$. If $e_1 \neq v_{0}v_{2n-1}'$ and $j = 0$, then $e_2 = v_iv_{2i-1}'$ and $e_3 = v_iv_{2i+1}'$. Hence, $\varphi(e_1)\varphi(e_2) = u_{3i}u_{3i-1} \in E(\widehat{C_{3n}})$ and
$\varphi(e_1)\varphi(e_3) = u_{3i}u_{3i+1} \in E(\widehat{C_{3n}}).$

Suppose now that $e_1 \neq v_{0}v_{2n-1}'$ and $j = 1$. Then $e_2 = v_{i}v_{2i}'$ and $e_3 = v_{i+1}v_{2i+1}'$. If $e_2 = v_0v_{2n-1}$, then $e_1 = v_0v_1$, a contradiction. If $e_3 = v_0v_{2n-1}$, then $e_1 = v_{n-1}v_{2n-1}$ and $e_2 = v_{n-1}v_{2n-2}'$. Thus,$\varphi(e_1)\varphi(e_2) = u_{3n-2}u_{3n-3}$ and $\varphi(e_1)\varphi(e_3) = u_{3n-2}u_{3n-1} \in E(\widehat{C_{3n}})$. Finally, if $e_2, e_3 \neq v_0v_{2n-1}$, then $\varphi(e_1)\varphi(e_2) = u_{3i+1}u_{3i} \in E(\widehat{C_{3n}})$ and $\varphi(e_1)\varphi(e_3) = u_{3i+1}u_{3i+1}$.

Therefore, $\varphi$ is a graph homomorphism. Similarly, if $\varphi(e)\varphi(e') \in E(\widehat{C_{3n}})$, then $ee' \in E(\Gamma_2(O_n))$.
    
\end{proof}

Let $S_n$ denote the {\it star graph} with vertex set $\{v,v_1,\dots,v_{n}\}$ such that $v$ is adjacent with the leaves $v_1,\dots,v_n$.
Let $\widehat{S_n}$ be the graph constructed from the star $S_n$ by adding a new vertex $u$ such that $u$ is adjacent by mean of two edges with the leaves of $S_n$, see Figure~\ref{fig:conestargraph}.

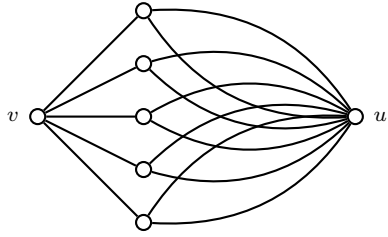
\begin{figure}[h!]
    \centering
    \begin{tikzpicture}[thick,scale=0.7]
    	\tikzstyle{every node}=[minimum width=0pt, inner sep=2pt, circle]
        \draw (-2,0) node[draw,label=180:{\footnotesize $v$}] (v) {};
        \draw (0,2) node[draw] (v1) {};
        \draw (0,1) node[draw] (v2) {};
        \draw (0,0) node[draw] (v3) {};
        \draw (0,-1) node[draw] (v4) {};
        \draw (0,-2) node[draw] (v5) {};
        \draw (4,0) node[draw,label=0:{\footnotesize $u$}] (u) {};
        \draw (v) -- (v1);
        \draw (v) -- (v2);
        \draw (v) -- (v3);
        \draw (v) -- (v4);
        \draw (v) -- (v5);
        \draw (u) edge[bend left] (v1);
        \draw (u) edge[bend left] (v2);
        \draw (u) edge[bend left] (v3);
        \draw (u) edge[bend left] (v4);
        \draw (u) edge[bend left] (v5);
        \draw (u) edge[bend right] (v1);
        \draw (u) edge[bend right] (v2);
        \draw (u) edge[bend right] (v3);
        \draw (u) edge[bend right] (v4);
        \draw (u) edge[bend right] (v5);
    \end{tikzpicture}
    \caption{The $\widehat{S_5}$ graph.}
    \label{fig:conestargraph}
\end{figure}

\begin{lemma}\label{lem:dualCnisSn}
    For $n\geq3$, the dual graph of $\widehat{C_n}$ is isomorphic to $\widehat{S_n}$.
\end{lemma}
\begin{proof}
Assume that the vertex set of $\widehat{C_n}$ is $\{u_0, \dots, u_{n-1}, u_0', \dots, u_{n-1}'\}$, such that the subset $\{u_0,\dots, u_{n-1}\}$ induces the cycle $C_n$ and $\{u_0', \dots, u_{n-1}'\}$ induces $\widehat{K_n}$.

On the other hand, assume that the vertex set of $\widehat{S_n}$ is $\{v, v_0, \dots, v_{n-1}, v'\}$, such that $\{v, v_0, \dots, v_{n-1}\}$ induces the star $\widehat{S_n}$, where $v$ has degree $n$, and $v'$ is adjacent twice to each $v_i$.

Let $G$ denote the dual graph of $\widehat{C_n}$. Let $V(G)$ be the vertex set of $G$, where $V(G) = \{f, f_0, \dots, f_{n-1}, f'\}$, such that $f$ corresponds to the face determined by the cycle $C_n$, $f_i$ corresponds to the face determined by the subgraph induced by $\{u_i', u_i, u_{i+1}\}$, and $f'$ corresponds to the outer face of $\widehat{C_n}$.

We observe that $f$ has degree $n$ and is adjacent to each $f_i$ exactly once. Moreover, each $f_i$ has degree $3$; besides being adjacent to $f$, it is adjacent to $f'$ twice.

Let us define $\varphi \colon V(G) \to V(\widehat{S_n})$ as follows:
\[
\varphi(x) = 
\begin{cases}
    v & \text{if $x = f$},\\
    v' & \text{if $x = f'$},\\
    v_i & \text{if $x = f_i$}.
\end{cases}
\]

Then $\varphi$ is a bijection, and it is easy to verify that it is a graph isomorphism.
\end{proof}

Note the Laplacian matrix $L\left(\widehat{S_n}\right)$ of $\widehat{S_n}$ is the $(n+2)\times(n+2)$ matrix
\[
\kbordermatrix{
       &    v_1 &    v_2 & \cdots &    v_n &      v &      u \\
   v_1 &      3 &      0 & \cdots &      0 &     -1 &     -2 \\
   v_2 &      0 &      3 & \cdots &      0 &     -1 &     -2 \\
\vdots & \vdots & \vdots & \ddots & \vdots & \vdots & \vdots \\
   v_n &      0 &      0 & \cdots &      3 &     -1 &     -2 \\
     v &     -1 &     -1 & \cdots &     -1 &      n &      0 \\
     u &     -2 &     -2 & \cdots &     -2 &      0 &     2n \\
         }.
\]


\begin{lemma}\label{lem:SNFSn}
    For $n\geq3$, the SNF of $L\left(\widehat{S_n}\right)$ is $\diag(1,1,3,\dots,3,6n,0)$.
\end{lemma}
\begin{proof}
    Let us start from $L\left(\widehat{S_n}\right)$ by subtracting the $n$-{\it th} row to the first $n-1$ rows, thus we obtain the matrix
    \[
    \begin{bmatrix}
      3 &      0 & \cdots &      0 &     -3 &      0 &      0 \\
      0 &      3 & \cdots &      0 &     -3 &      0 &      0 \\
 \vdots & \vdots & \ddots & \vdots & \vdots & \vdots & \vdots \\
      0 &      0 & \cdots &      3 &     -3 &      0 &      0 \\
      0 &      0 & \cdots &      0 &      3 &     -1 &     -2 \\
     -1 &     -1 & \cdots &     -1 &     -1 &      n &      0 \\
     -2 &     -2 & \cdots &     -2 &     -2 &      0 &     2n \\
    \end{bmatrix}.
    \]
    After adding the first $n-1$ columns to the $n$-{\it th} column, we obtain
    \[
    \begin{bmatrix}
      3 &      0 & \cdots &      0 &      0 &      0 &      0 \\
      0 &      3 & \cdots &      0 &      0 &      0 &      0 \\
 \vdots & \vdots & \ddots & \vdots & \vdots & \vdots & \vdots \\
      0 &      0 & \cdots &      3 &      0 &      0 &      0 \\
      0 &      0 & \cdots &      0 &      3 &     -1 &     -2 \\
     -1 &     -1 & \cdots &     -1 &     -n &      n &      0 \\
     -2 &     -2 & \cdots &     -2 &    -2n &      0 &     2n \\
    \end{bmatrix}.
    \]
    By subtracting the $(n-1)$-{\it th} column to the first $n-2$ columns, we obtain
    \[
    \begin{bmatrix}
      3 &      0 & \cdots &      0 &      0 &      0 &      0 &      0 \\
      0 &      3 & \cdots &      0 &      0 &      0 &      0 &      0 \\
 \vdots & \vdots & \ddots & \vdots & \vdots & \vdots & \vdots & \vdots \\
      0 &      0 & \cdots &      3 &      0 &      0 &      0 &      0 \\
     -3 &     -3 & \cdots &     -3 &      3 &      0 &      0 &      0 \\
      0 &      0 & \cdots &      0 &      0 &      3 &     -1 &     -2 \\
      0 &      0 & \cdots &      0 &     -1 &     -n &      n &      0 \\
      0 &      0 & \cdots &      0 &     -2 &    -2n &      0 &     2n \\
    \end{bmatrix}.
    \]
    Now, adding the first $n-2$ rows to the $(n-1)$-{\it th} row, we obtain the matrix
    \[
    \diag(3,\dots,3)\oplus
    \begin{bmatrix}
      3 &      0 &      0 &      0 \\
      0 &      3 &     -1 &     -2 \\
     -1 &     -n &      n &      0 \\
     -2 &    -2n &      0 &     2n \\
    \end{bmatrix}.
    \]
    By adding the $(n+1)$-{\it th} column to the $n$-{\it th} column, we obtain
    \[
    \diag(3,\dots,3)\oplus
    \begin{bmatrix}
      3 &      0 &      0 &      0 \\
      0 &      2 &     -1 &     -2 \\
     -1 &      0 &      n &      0 \\
     -2 &    -2n &      0 &     2n \\
    \end{bmatrix}.
    \]
    By adding the $n$-{\it th} column to the $(n+2)$-{\it th} column, we obtain
    \[
    \diag(3,\dots,3)\oplus
    \begin{bmatrix}
      3 &      0 &      0 &      0 \\
      0 &      2 &     -1 &      0 \\
     -1 &      0 &      n &      0 \\
     -2 &    -2n &      0 &      0 \\
    \end{bmatrix}.
    \]
    By adding two times the $(n+1)$-{\it th} column to the $n$-{\it th} column, we obtain
    \[
    \diag(3,\dots,3)\oplus
    \begin{bmatrix}
      3 &      0 &      0 &      0 \\
      0 &      0 &     -1 &      0 \\
     -1 &     2n &      n &      0 \\
     -2 &    -2n &      0 &      0 \\
    \end{bmatrix}.
    \]
    By adding $n$ times the $n$-{\it th} row to the $(n+1)$-{\it th} row, we obtain
    \[
    \diag(3,\dots,3)\oplus
    \begin{bmatrix}
      3 &      0 &      0 &      0 \\
      0 &      0 &     -1 &      0 \\
     -1 &     2n &      0 &      0 \\
     -2 &    -2n &      0 &      0 \\
    \end{bmatrix}.
    \]
    Finally, by adding three times the $(n+1)$-{\it th} row to the $(n-1)$-{\it th} row, and subtracting two times the $(n+1)$-{\it th} row to the $(n+2)$-{\it th} row, we obtain
    \[
    \diag(3,\dots,3)\oplus
    \begin{bmatrix}
      0 &     6n &      0 &      0 \\
      0 &      0 &     -1 &      0 \\
     -1 &     2n &      0 &      0 \\
      0 &    -6n &      0 &      0 \\
    \end{bmatrix}.
    \]
    From this, we can see that the Smith normal form of $L\left(\widehat{S_n}\right)$ is $\diag(1,1,3,\dots,3,6n,0)$.
\end{proof}

Now we state the main result of this section.

\begin{theorem}
The Smith normal form of the second Laplacian matrix of the simplicial annulus $O_n$ is ${\sf I}_{3n}\oplus3{\sf I}_{3n-2}\oplus\diag(18n,0)$.
\end{theorem}
\begin{proof}
    By Lemma~\ref{lem:secondauxiliarygraphOnisomorphicC3n}, $L^2(O_n)=L\left(\widehat{C_{3n}}\right)$.
    We know by Lemma~\ref{lem:dualCnisSn} that the dual graph of $\widehat{C_{3n}}$ is isomorphic to $\widehat{S_{3n}}$, then $K\left(\widehat{C_{3n}}\right)\simeq K\left(\widehat{S_{3n}}\right)$, which in turn are isomorphic to $\bigoplus_{i=1}^{3n-2}\mathbb{Z}_3\oplus\mathbb{Z}_{18n}$ by Lemma~\ref{lem:SNFSn}.
    Since $C_{3n}$ is a connected graph, then there is only one zero in the Smith normal form of $L\left(\widehat{C_{3n}}\right)$.
    And since $C_{3n}$ has $6n$ vertices, then the Smith normal form of $L\left(\widehat{C_{3n}}\right)$ is $I_{3n}\oplus3I_{3n-2}\oplus\diag(18n,0)$.
\end{proof}

It is important to say that computing the SNF of the first Laplacian matrix of the simplicial annulus seems to be much more complicated to be obtained, however, we observed that it is related with the Fibonacci sequence.

\section{$k$-trees}\label{sec:ktrees}

A $k$-{\it tree} is either a complete graph on $k$ vertices or a graph obtained from a smaller $k$-tree by adjoining a new vertex together with $k$ edges connecting it to a $k$-clique.
Regarding a $k$-tree as simplicial complex, all its cliques are allowed.
In this section we compute the SNF of the $k$-Laplacian matrices of $k$-trees.
It is interesting that a generalization of the celebrated formula obtained by Graham-Lovász-Pollak of the determinant of the distance matrix of trees to $k$-trees was obtained by computing the SNF of the $k$-distance matrices of $k$-trees in \cite{MR5026280}.
Finally in Subsection~\ref{sec:cospeccoinv}, we explore computationally how good are the SNF and the spectrum to distinguish $k$-trees.

Harary and Palmer, in \cite{MR228355,MR357214}, counted unlabeled 2-trees in 1968.
The enumeration of unlabeled $k$-trees for $k > 2$ was a long-standing unsolved problem
until the recent solution by Gainer-Dewar in \cite{MR3007180}, using the theory of combinatorial
species.
In \cite{MR3213312}, unlabeled $k$-trees are enumerated by first coloring the vertices using $k+1$ colors in such a way that the vertices of each $(k+1)$-clique use all $k+1$ colors.
See the sequences \href{https://oeis.org/A054581}{A054581} and \href{https://oeis.org/A370770}{A370770} in The On-Line Encyclopedia of Integer Sequences \cite{oeis}.
In \ref{sec:generation of k trees}, we include {\tt sagemath} code that generates the $k$-trees with up to $n$ vertices. 
And in \ref{sec:k Laplacian matrix}, there is a code to compute the SNF of the $d$-Laplacian matrix for $k$-trees.

\begin{figure}[h]
    \centering
    \begin{tabular}{c@{\extracolsep{2cm}}c}
        \begin{tikzpicture}[scale=0.7,thick]
		\tikzstyle{every node}=[minimum width=0pt, inner sep=2pt, circle]
			\draw (120:1) node[draw] (2) {};
			\draw (240:1) node[draw] (3) {};
            \draw (0:1) node[draw] (5) {};
			\draw (3)++(180:1)++(3) node[draw] (4) {};
            \draw (2)++(180:1)++(2) node[draw] (0) {};
			\draw (2)++(180:1)++(3) node[draw] (1) {};
			\draw (3)++(5)++(-60:1) node[draw] (6) {};
			\draw (2)++(5)++(60:1) node[draw] (7) {};
			\draw  (0) edge (2);
			\draw  (1) edge (2);
			\draw  (1) edge (3);
			\draw  (3) edge (5);
			\draw  (5) edge (7);
			\draw  (2) edge (7);
			\draw  (2) edge (5);
			\draw  (5) edge (6);
			\draw  (3) edge (6);
			\draw  (3) edge (4);
			\draw  (1) edge (4);
			\draw  (0) edge (1);
			\draw  (2) edge (3);
		\end{tikzpicture}
        & 
        \begin{tikzpicture}[scale=0.5,thick]
		\tikzstyle{every node}=[minimum width=0pt, inner sep=2pt, circle]
            \draw (0:1) node[draw] (0) {};
			\draw (120:1) node[draw] (1) {};
			\draw (240:1) node[draw] (2) {};
            \draw (0)++(0)++(60:1) node[draw] (3) {};
            \draw (0)++(0)++(-60:1) node[draw] (4) {};
            \draw (3)++(60:1)++(0:1) node[draw] (5) {};
            \draw (3)++(60:1)++(120:1) node[draw] (6) {};
            \draw (4)++(-60:1)++(0:1) node[draw] (7) {};
            \draw (4)++(-60:1)++(-120:1) node[draw] (8) {};
            \draw (1)++(120:1)++(0:-1) node[draw] (9) {};
            \draw (1)++(120:1)++(60:1) node[draw] (10) {};
            \draw (2)++(-120:1)++(0:-1) node[draw] (11) {};
            \draw (2)++(-120:1)++(-60:1) node[draw] (12) {};
            \draw (0) -- (1) -- (10) -- (9) -- (1) -- (2) -- (11) -- (12) -- (2) -- (0) -- (4) -- (8) -- (7) -- (4) -- (3) -- (5) -- (6) -- (3) -- (0);
		\end{tikzpicture}\\
        $T$ & $\Gamma_2(T)$
    \end{tabular}
    \caption{A 2-tree $T$ and its 2-adjacency graph $\Gamma_2(T)$.}
    \label{fig:2tree and auxliliary graph}
\end{figure}
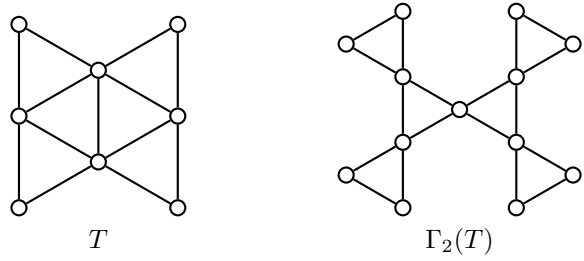


\begin{lemma}\label{lem:AuxGraphofTreeisblock}
    Let $T$ be a $k$-tree with $n$ vertices such that $n>k$.
    Then the $k$-{\it th} adjacency graph $\Gamma_k(T)$ of $T$ is a block graph with $(n-k)$ blocks consisting of $(k+1)$-cliques.
\end{lemma}
\begin{proof}
Let $T_{n+1}$ be a $k$-tree with $n+1$ vertices constructed from the $k$-tree $T_n$ with $n$ vertices by adjoining, to a $k$-clique $\alpha\in T_n$, a new vertex $v$ together with $k$ edges connecting each vertex in $\alpha$ with $v$.
Thus $\Gamma_k(T_{n+1})$ is constructed from $\Gamma_k(T_n)$ by adding a $k$-clique such that each of its new vertices are adjacent with the vertex associated with $\alpha$.
Moreover, $\Gamma_k(T_n)$ has $k(n-k)+1$ vertices.
Then the result turns out.
\end{proof}

\begin{theorem}\label{teo:ktrees}
    Let $T$ be a $k$-tree with $n$ vertices such that $n>k$.
    Then, the Smith normal form of $L^k(T)$ is ${\sf I}_{n-k}\oplus(k+1){\sf I}_{(n-k)(k-1)}\oplus[0]$.
\end{theorem}
\begin{proof}
    By Lemma~\ref{lem:AuxGraphofTreeisblock}, the $k$-{\it th} adjacency graph $\Gamma_k(T)$ of $T$ is a block graph with $(n-k)$ blocks consisting of $(k+1)$-cliques, then $K(\Gamma_k(T))\simeq \oplus_{i=1}^{n-k}\mathbb{Z}_{k+1}^{k-1}$.
    Therefore, the diagonal of the SNF of $L(\Gamma_k(T))$ contains $k(n-k)-(n+k)$ entries equal to $k+1$.
    And, since $\Gamma_k(T)$ is connected, then the diagonal of the SNF of $L(\Gamma_k(T))$ contains only one zero entry, and the other $(n-k)$ diagonal entries are ones.
    Finally, by Lemma~\ref{lem:LmatofAuxandComplexareequal}, the SNFs of $L(\Gamma_k(T))$ and $L^k(T)$ coincide.
\end{proof}

The following question arises: is it possible to describe the SNF of $L^d(T)$ for $d\in\{1,\dots,k-1\}$?
This is not an easy question, as the reader might note by exploring the following code, which uses the functions defined in \ref{sec:generation of k trees} and \ref{sec:k Laplacian matrix}.

\begin{lstlisting}
for k in range(2,6):
    print("\n\n********************* k = " + str(k) + "*********************")
    n = k
    for L in k_trees(k,k+5):
        if len(L) == 1:
            print("*** There is one " + str(k) + "-trees with " + str(n) + " vertices")
        else:
            print("*** There are " + str(len(L)) + " " + str(k) + "-trees with " + str(n) + " vertices")
        for l in L:
            l.show()
            print("graph6 code: " + l.graph6_string())
            for i in range(1,k+1):
                d_laplacian_matrix(l,i)
        n = n + 1
\end{lstlisting}


In particular, computing the SNF of the 1-Laplacian matrix of 2-trees which are outerplanar graphs is not an easy task since computing the SNF of the Laplacian matrices of outerplanar graphs is difficult. 
For instance, the structure of the sandpile group of few subfamilies of the outerplanar graphs have been completely established, see for example \cite{MR3442497,MR4023161,MR3479466}.
Also, the Tutte polynomial and the number of spanning trees of an infinite families of outerplanar, small-world and self-similar graphs were obtained in \cite{MR3046135,MR3083114}.
In \cite{AVi}, it is given a methodology to compute the algebraic structure of the sandpile groups of outerplanar graphs by evaluating the indeterminates of the critical ideals of the weak dual at the lengths of the cycles bounding the interior faces of the plane graph.
Despite these advances, there is still work to be done to completely understand the SNF of the Laplacian matrices of outerplanar graphs, and so the SNF of the 1-Laplacian matrix of 2-trees.

\subsection{Cospectral and coinvariant $k$-trees}\label{sec:cospeccoinv}

In this subsection, we would like to explore if we could use all the SNF of all the Laplacian matrices to distinguish $k$-trees.
As we have seen in Theorem~\ref{teo:ktrees} one matrix is not enough since we already see that the SNF of the $L^k$ matrix is the same for the $k$-trees with $n$ vertices.

Recall that two non-isomorphic graphs with the same spectrum with respect the Laplacian matrix are called {\it cospectral graphs}.
And a graph $G$ is said to be {\it determined} by the Laplacian spectrum, if there is no graph not isomorphic to $G$ whose Laplacian spectrum is the same Laplacian spectrum of $G$.
On the other hand, two non-isomorphic graphs with the same SNF with respect the Laplacian matrix are called {\it coinvariant graphs}.
And a graph $G$ is said to be {\it determined} by the SNF of $L(G)$, if there is no graph not isomorphic to $G$ whose SNF of its Laplacian matrix have the same SNF that $L(G)$.
The {\it graph isomorphism problem} asks to determine whether two graphs are isomorphic, that is to say, whether the adjacency matrices of the two graphs are permutation similar.
Motivated by the graph isomorphism problem, it is of interest what fraction of all graphs is uniquely determined by its spectrum or by its SNF. 
Haemers conjectured that the fraction of graphs on $n$ vertices with a $M$-cospectral mate tends to zero as $n$ tends to infinity. 
Several numerical studies have been performed for graphs, for example, a numerical study for $n\leq 9$ was given by Godsil and McKay \cite{gm}, for $n = 10, 11$ by Haemers and Spence \cite{HS2004} and for $n=12$ by Brouwer and Spence \cite{bs}. 
Aouchiche and Hansen \cite{ah} presented computational results in which they studied cospectrality for the distance, distance Laplacian and distance signless Laplacian matrices of all the connected graphs on up to 10 vertices. 
Pinheiro, Souza and Trevisan \cite{MR4061004} provided some numerical evidence that the complementary spectrum of a graph distinguishes more graphs than other standard graph spectra, but they also showed that it is hard to compute the complementary spectrum.
In \cite{aa}, enumeration results were obtained for the spectrum and the SNF for the matrices $A$, $L$, $Q$, $D$, $D^L$ and $D^Q$.
Recently in \cite{az}, the enumeration results were obtained for the matrices $A^{\tr}$, $A^{\tr}_+$, $D^{\deg}$ and $D^{\deg}_+$.

Let $\mathcal{T}_{k,n}$ denote the set of $k$-trees with $n$ vertices.
In Table~\ref{tab:ktreeswithnvertices}, we show the number $|\mathcal{T}_{k,n}|$ of $k$-trees with $n$ vertices. 
See the sequences \href{https://oeis.org/A054581}{A054581} and \href{https://oeis.org/A370770}{A370770} in The On-Line Encyclopedia of Integer Sequences \cite{oeis}.
In \ref{sec:generation of k trees}, we include {\tt sagemath} code that generates the $k$-trees with up to $n$ vertices.

\begin{table}[h]
    \centering
    \footnotesize
    \begin{tabular}{c|cccccccccccccc}
        $k\backslash n$ & 2 & 3 & 4 & 5 & 6 & 7  & 8  & 9   & 10  & 11    & 12    & 13     \\
        \hline
                      2 & 1 & 1 & 1 & 2 & 5 & 12 & 39 & 136 & 529 & 2,171 & 9,368 & 41,534 \\
                      3 &   & 1 & 1 & 1 & 2 &  5 & 15 &  58 & 275 & 1,505 & 9,003 & 56,931 \\
                      4 &   &   & 1 & 1 & 1 &  2 &  5 &  15 &  64 &   331 & 2,150 & 15,817 \\
                      5 &   &   &   & 1 & 1 &  1 &  2 &   5 &  15 &    64 &   342 &  2,321 \\
                      6 &   &   &   &   & 1 &  1 &  1 &   2 &   5 &    15 &    64 &    342 \\
                      7 &   &   &   &   &   &  1 &  1 &   1 &   2 &     5 &    15 &     64 \\
    \end{tabular}
    \caption{The number $|\mathcal{T}_{k,n}|$ of $k$-trees with $n$ vertices.}
    \label{tab:ktreeswithnvertices}
\end{table}

We say that two non-isomorphic $k$-trees are {\it full coinvariant} if for each $d\in\{1,\dots,k\}$, the SNF of the $d$-Laplacian matrices of both $k$-trees are the same. 
In such case, those $k$-trees are called {\it full coinvariant mates}.
Let $\mathcal{T}^{in}_{k,n}$ denote the set of $k$-trees with $n$ vertices which have a full coinvariant mate.
In Table~\ref{tab:coinvariantktrees}, we show the number $|\mathcal{T}^{in}_{k,n}|$ of $k$-trees with $n$ vertices which have a full coinvariant mate.
Note that the number of $k$-trees with $n$ vertices determined by the SNF of its $d$-Laplacian matrices coincides with $|\mathcal{T}_{k,n}|-|\mathcal{T}^{in}_{k,n}|$.

\begin{table}[h]
    \centering
    \footnotesize
    \begin{tabular}{c|cccccccccccccc}
        $k\backslash n$ & 2 & 3 & 4 & 5 & 6 & 7  & 8  & 9   & 10  & 11    & 12    & 13     \\
        \hline
                      2 & 0 & 0 & 0 & 0 & 2 &  7 & 31 & 122 & 509 & 2,141 & 9,330 & 41,472 \\
                      3 &   & 0 & 0 & 0 & 0 &  0 &  0 &   0 &   0 &     0 &     2 &      6 \\
                      4 &   &   & 0 & 0 & 0 &  0 &  0 &   0 &   0 &     0 &     0 &      0 \\
                      5 &   &   &   & 0 & 0 &  0 &  0 &   0 &   0 &     0 &     0 &      0 \\
                      6 &   &   &   &   & 0 &  0 &  0 &   0 &   0 &     0 &     0 &      0 \\
                      7 &   &   &   &   &   &  0 &  0 &   0 &   0 &     0 &     0 &      0 \\
    \end{tabular}
    \caption{The number $|\mathcal{T}^{in}_{k,n}|$ of full coinvariant $k$-trees with $n$ vertices.}
    \label{tab:coinvariantktrees}
\end{table}

It is interesting to note that the number of full coinvariant trees reduces as $k$ grows.
This suggest the following question: is there exists an $k$ such that all $k$-trees with $n$ vertices are determined by the SNF of its $d$-Laplacian matrices?

Analogously, we say that two non-isomorphic $k$-trees are {\it full cospectral} if for each $d\in\{1,\dots,k\}$, the spectrum of the $d$-Laplacian matrices of both $k$-trees are the same. 
In such case, those $k$-trees are called {\it full cospectral mates}.
Let $\mathcal{T}^{sp}_{k,n}$ denote the set of $k$-trees with $n$ vertices which have a full cospectral mate.
In Table~\ref{tab:coinvariantktrees}, we show the number $|\mathcal{T}^{in}_{k,n}|$ of $k$-trees with $n$ vertices which have a full cospectral mate.
Note that the number of $k$-trees with $n$ vertices determined by the spectrum of its $d$-Laplacian matrices coincides with $|\mathcal{T}_{k,n}|-|\mathcal{T}^{sp}_{k,n}|$.

\begin{table}[h]
    \centering
    \footnotesize
    \begin{tabular}{c|cccccccccccccc}
        $k\backslash n$ & 2 & 3 & 4 & 5 & 6 & 7  & 8  & 9   & 10  & 11    & 12    & 13     \\
        \hline
                      2 & 0 & 0 & 0 & 0 & 0 &  0 &  0 &   2 &   0 &     2 &    24 &     68 \\
                      3 &   & 0 & 0 & 0 & 0 &  0 &  0 &   0 &   0 &     0 &     0 &      2 \\
                      4 &   &   & 0 & 0 & 0 &  0 &  0 &   0 &   0 &     0 &     0 &      0 \\
                      5 &   &   &   & 0 & 0 &  0 &  0 &   0 &   0 &     0 &     0 &      0 \\
                      6 &   &   &   &   & 0 &  0 &  0 &   0 &   0 &     0 &     0 &      0 \\
                      7 &   &   &   &   &   &  0 &  0 &   0 &   0 &     0 &     0 &      0 \\
    \end{tabular}
    \caption{The number $|\mathcal{T}^{sp}_{k,n}|$ of full cospectral $k$-trees with $n$ vertices.}
    \label{tab:cospectralktrees}
\end{table}

Clearly, the performance of the spectrum to distinguish $k$-trees is much better than the SNF.
Again a question arises: is there exists an $k$ such that all $k$-trees with $n$ vertices are determined by the spectrum of its $d$-Laplacian matrices?

\section{The Graham-Lovász-Pollak matrix}\label{sec:GLP}


Hou and Woo \cite{HW} extended the Graham-Lovász-Pollak \cite{GL,GP} celebrated formula $$\det(D(T_{n+1}))=(-1)^nn2^{n-1},$$ for any tree $T_{n+1}$ with $n+1$ vertices to the SNF of the distance matrix, whose proof uses the following matrix introduced by Graham, Lovász and Pollak in \cite{GL,GP}.

\begin{definition}\label{def:P}
    Let $T$ be a tree with $n$ vertices whose vertex set is $\{v_1,\dots,v_n\}$.
    Let $P(T)$ denote the $n \times n$ matrix having $1$ in entry $(i,j)$ if vertex $v_i$ lies on the unique $v_jv_n$-path, and $0$ otherwise. 
\end{definition}

Now, let us recall Hou and Woo result.

\begin{theorem}\cite[Theorem 3]{HW}\label{thm:HouWoo}
Let $T_{n+1}$ be a tree with $n+1$ vertices, then
$\SNF(D(T_{n+1}))=\sf{I}_2\oplus 2\sf{I}_{n-2}\oplus (2n)$.
\end{theorem}

From the context of sandpile groups, see \cite[Section 4.5]{klivans}, it is known that $\SNF(L(T))=\diag(1,\dots,1,0)$.
This follows since the product of the positive invariant factors of the Laplacian matrix of a connected graph equals the number of spanning trees, but no explicit matrix is known to obtain the SNF of $L(T)$.
It is interesting that the matrix introduced by Graham, Lovász and Pollak in \cite{GL,GP} also can be used to deduce the SNF of the Laplacian matrix of trees, as we will see in Theorem~\ref{teo:main}.

Graham, Lovász and Pollak noticed that $P(T)$ is unimodular, we recall this result next.

\begin{lemma}\label{lema:P}
    Let $T$ be a tree with $n$ vertices whose vertex set is $\{v_1,\dots,v_n\}$.
    Let $P=P(T)$ as Definition~\ref{def:P}.
    And let $Q$ be the $n\times n$ matrix such that it has $1$ in entry $(i,i)$, $-1$ in entry $(i,j)$ if $v_i$ lies on the unique $v_jv_n$-path and $v_i$ is adjacent to $v_j$, and $0$ otherwise.
    Then
    \begin{enumerate}[(i)]
        \item $\det(P) = 1$, and
        \item $Q$ is the inverse of $P$.
    \end{enumerate}
\end{lemma}
\begin{proof}
$(i)$ Let $T$ be a tree. We will proceed by induction on the order of $T$. If $n = 2$, then
\[
P = \begin{bmatrix}
    1 & 0 \\
    1 & 1
\end{bmatrix},
\]
which implies $\det(P) = 1$.
Suppose the result holds for trees with order $n > 1$. 
If the order of $T$ is $n + 1$, then we can choose a leaf of $T$, say $v_i$ such that $v_i \neq v_n$. Let $v_t$ be adjacent to $v_i$. 
Note that $P_{ik} = 1$ if and only if $k = t$, then $\det(P) = (-1)^{i+i}\det(P[i])$, where $P[i]$ is the matrix $P$ of the tree $T\setminus\{v_i\}$, thus $\det(P) = 1$.

$(ii)$ Let $i,j \in \{1,\dots, n\}$, with $i \neq j$. 
We compute $(PQ)_{ij} = \sum_{k=1}^nP_{ik}Q_{kj}$. 
Let $W = (v_j, v_t, \dots, v_n)$ be the unique $v_jv_n$-path. 
By definition, $Q_{tj} = -1$ and $Q_{jj} = 1$. 
Let $k \neq j, t$, if $Q_{kj} = -1$, then $v_k \in W$ and $v_kv_j \in E(G)$, thus $v_k = v_t$, which is a contradiction. 
Hence
\[
Q_{kj} = \begin{cases}
    1 & \text{if } k = j \\
    -1 & \text{if } k = t \\
    0 & \text{otherwise}.
\end{cases}
\]
Therefore $(PQ)_{ij} = P_{ij} - P_{it}$. 
If $P_{ij} = 1$, then $v_i \in W$. 
Since $i \neq j$, then $v_i$ belongs to the unique $v_tv_n$-path, so $P_{it} = 1$. 
If $P_{ij} = 0$, then $P_{it} = 1$ implies $P_{ij} = 1$, which is a contradiction.
Therefore $(PQ)_{ij} = 0$ whenever $i \neq j$.
If $i = j$, then $(PQ)_{ij} = \sum_{k=1}^nP_{ik}Q_{ki}$. 
Let $Y = (v_i, v_s \dots, v_n)$ be the unique $v_iv_n$-path. 
Similarly,
\[
Q_{ki} = \begin{cases}
    1 & \text{if } k = i \\
    -1 & \text{if } k = s \\
    0 & \text{otherwise}.
\end{cases}
\]
Then $(PQ)_{ij} = P_{ii} - P_{is} = 1 - 0 = 1$.\\ 
Therefore $PQ = I_n$, implying $P^{-1} = Q$. \\
\end{proof}

Now we use matrix $P(T)$ to deduce the SNF of the Laplacian matrix of trees.

\begin{theorem}\label{teo:main}
Let $T$ be a tree with $n$ vertices $v_1, \dots, v_n$.
Let $P=P(T)$ as Definition~\ref{def:P}.
Then, $PL(T)P^t = \diag(1,1,\dots,1,0)$.
\end{theorem}

\begin{proof} 
We will show that $L(T)P^t = Q \diag(1,1,\dots,1,0)$. 
Let $i, j \in \{1,\dots, n\}$. 
Without loss of generality, assume $\{v_1,\dots,v_t\}$ are the vertices adjacent to $v_i$. 
In the following, to simplify the notation $L=L(T)$.
Then
\begin{equation}
\label{a}
(LP^t)_{ij} = \sum_{k=1}^nL_{ik}P_{kj}^t = \sum_{k=1}^nL_{ik}P_{jk}.
\end{equation}

{\bf Case 1:} $j = n$. 
As $P_{nk} = 1$ for all $k$, then Equation $(\ref{a})$ becomes $\sum_{k=1}^nL_{ik} = 0$.

{\bf Case 2:} $i \neq j$ and $j\neq n$. Consider the following subcases.

{\bf Subcase 2.1:} $P_{ji} = 0$ and $v_iv_j \notin E(G)$. 
We have $Q_{ij} = 0$. 
Furthermore, Equation $(\ref{a})$ becomes $L_{ii}P_{ji} - \sum_{k=1}^tP_{jk} = - \sum_{k=1}^tP_{jk}$. 
Let $k \in \{1,\dots, t\}$ and $W = (v_k, v_s, \dots, v_n)$ be the unique $v_kv_n$-path. 
If $P_{jk} = 1$, then $v_j \in W$, and if $v_s = v_i$, then $P_{ji} = 1$, which is a contradiction. 
If $v_i \notin W$, then $\{v_i\} \cup W$ is the unique $v_iv_n$-path and $v_j \in \{v_i\} \cup W$, which is a contradiction. 
Therefore $P_{jk} = 0$ for all $k \in \{1,\dots, t\}$, thus Equation $(\ref{a})$ becomes $- \sum_{k=1}^tP_{jk} = 0 = Q_{ij}$.

{\bf Subcase 2.2:} $P_{ji} = 0$ and $v_iv_j \in E(G)$. 
If $W$ is the $v_iv_n$-path, then $v_j \notin W$, hence $\{v_j\}\cup W$ is the unique $v_jv_n$-path. 
Thus $Q_{ij} = -1$. 
Let us assume that $\{v_j\}\cup W = \{v_j,v_i,v_s,\dots, v_n\}$ for some $s \in \{1,\dots,t\}$. 
Then for every $k \in \{1,\dots,t\}\setminus\{s\}$, the $v_kv_n$-path is $\{v_k\} \cup W$, which does not contain $v_j$. 
Therefore, $P_{jk} = 0$ and $P_{js} = 1$, thus Equation $(\ref{a})$ becomes $L_{ii}P_{ji} - \sum_{k=1}^{t}L_{ik}P_{jk} = -P_{jj} = -1 = Q_{ij}$.

{\bf Subcase 2.3:} $P_{ji} = 1$ and $v_iv_j \in E(G)$. 
We have $Q_{ij} = 0$, since $W = (v_i,v_j,\dots,v_n)$ is the $v_iv_n$-path and $v_i \notin P\setminus\{v_i\}$, which is the $v_jv_n$-path. 
Moreover, $v_j$ is in all paths from vertices adjacent to $v_i$ towards $v_n$, so $P_{jk} = 1$ for all $k \in \{1,\dots,t\}$. 
Therefore, Equation $(\ref{a})$ becomes $L_{ii}P_{ji} - \sum_{k=1}^nP_{jk} = L_{ii} - L_{ii} = 0 = Q_{ij}$.

{\bf Subcase 2.4:} $P_{ji} = 1$ and $v_iv_j \notin E(G)$. 
We have $q_{ij} = 0$.
Let $W = (v_i,v_s, \dots, v_j, \dots, v_n)$ be the $v_iv_n$-path. 
Notice that $v_j$ belongs to all paths from vertices adjacent to $v_i$ towards $v_n$. 
Hence Equation $(\ref{a})$ becomes $L_{ii}P_{ji} - \sum_{k=1}^nP_{jk} = L_{ii} - L_{ii} = 0 = Q_{ij}$.

{\bf Case 3:} $i = j$, $j \neq n$. 
By definition, $Q_{ii} = 1$. 
If $W = (v_i,\dots, v_n)$ is the $v_iv_n$-path, then $P_{ik} = 1$ if $k \in \{1,\dots, t\}\setminus\{s\}$ and $P_{is} = 0$. 
Thus Equation $(\ref{a})$ becomes $L_{ii}P_{ji} - \sum_{k=1}^nP_{jk} = L_{ii} - (L_{ii}-1) = 1 = Q_{ii}$. 

Therefore, $L(T)P^t = Q \diag(1,1,\dots,1,0)$. 
Hence, by Lemma~\ref{lema:P}, 
\[
PL(T)P^t = \diag(1,1,\dots,1,0).
\]
\end{proof}

Recall that a {\it block graph} or {\it clique tree} is a graph in which every block is a clique. 
Thus, if $G$ is a block graph, then we can extend Definition~\ref{def:P} to $G$ as follows $P=P(G)$ is the $|G|\times |G|$ matrix with $P_{i,j}=\begin{cases}
    1, & \text{if }v_i\text{ lies on the shortest }v_jv_n-\text{path}\\
    0, & \text{otherwise}
\end{cases}$

\begin{theorem}\label{teo:main-block-L}
    Let $G$ be a connected block graph with $n$ vertices, where $v_n$ is not a cut vertex in a maximal clique with only one cut vertex 
    (we call this the root vertex), and let $P=P(G)$. Also, let us assume that the vertices are ordered such that the vertices in any given clique are consecutive, with the condition that a cut-vertex belongs (with respect to this labeling) to the block closer to the root. 
    Then, $$PL(G)P^t =  \bigoplus_{i=1}^k B_{s_i}\oplus [0], \text{ where } B_s = sI_{s-1}-J_{s-1}\text{ for } l\geq 3, \text{and } B_2=I_1=[1].$$
\end{theorem}

\begin{proof}
    Let $P=\begin{pmatrix}
        \mathbf{r}_1\\
        \mathbf{r}_2\\
        \vdots\\
        \mathbf{r}_{n-1}\\
        \mathbf{r}_n
    \end{pmatrix}$ and
    note that, by definition, the row vector $\mathbf{r}_i$ is the indicator vector of the set $U_i=\{v_j\ |\ v_i\text{ is in the shortest path from }v_j\text{ to }v_n \}\subset V(G)$. Now, let us set $M=PL(G)P^t$. We will prove the result by explicitly computing the entries $M_{i,j}=\mathbf{r}_iL(G)\mathbf{r}_j^t$.\\
    We will use the fact that given two vectors $\mathbf{x},\mathbf{y}\in \{0,1\}^{V(G)}$, then $$\displaystyle\mathbf{x}L(G)\mathbf{y}^t=\sum_{(u,v)\in E(G)}(x_u-x_v)(y_u-y_v)$$
    Moreover, if $\mathbf{x}$ and $\mathbf{y}$ are indicator vectors of $X,Y\subseteq V(G)$, respectively. Then,
    $$\displaystyle\mathbf{x}L(G)\mathbf{y}^t=e(X\cap Y,(X\cup Y)^c)-e(X\setminus Y,Y\setminus X),$$
    where $e(X,Y)$ is the number of edges with one endpoint in $X$ and the other in $Y$.\\
    \textbf{Case 1: Last row and column.} Since $\mathbf{r}_n=\mathbf{1}$, $\mathbf{1}L(G)=\mathbf{0}$, and $L(G)\mathbf{1}^t=\mathbf{0}^t$. Then 
    \[M_{i,n}=\mathbf{r}_1L(G)\mathbf{1}^t=0\text{ and }M_{n,j}=\mathbf{1}L(G)\mathbf{r}_j=0\text{ for all }1\leq i,j\leq n\]
    This yields the $(0)$ block in $M$.\\
    \textbf{Case 2: Diagonal entries in a given block.} If  $v_i$ is not a cut-vertex in a block $B_s\subset G$ with $s$ vertices. Then $U_i=\{v_i\}$, and $$M_{i,i}=\mathbf{r}_iL(G)\mathbf{r}_i^t=e(U_i,U_i^c)=e(\{v_i\},V(G)\setminus \{v_i\})=\deg(v_i)=s-1.$$ 
    A block has one or more cut-vertices. Let us name the cut-vertex closest to $v_n$ as \textit{the exit vertex} of the block. Thus, the above also holds if $v_i$ is a non-exit cut-vertex. On the other hand, if $v_i$ is the exit vertex, then $M_{i,i}=e(U_i,U_i^c)=s'-1$ where $B_{s'}$ is the other block to which $v_i$ belongs that is closest to $v_n$ (note that $v_i$ can belong to more than $2$ blocks).\\ 
    \textbf{Case 3: Non-diagonal entries in a given block.} If $i\neq j$ and $v_i$ and $v_j$ are in the same block $B_s$. Moreover, assume that none of these vertices is the exit vertex, then $U_i\cap U_j=\emptyset$ and $M_{i,j}=-e(U_i,U_j)=-1$. If $v_j$ is the exit vertex, then $U_i\setminus U_j=\emptyset$ and $M_{i,j}=e(U_i,U_j^c)=0$.\\
    Thus, case 2 and 3 gives us the blocks $B_{s_i}$ in $M$.\\
    \textbf{Case 4: Other entries} Now, we have to make sure that the other entries are indeed equal to $0$. First, similarly to the previous case, if $v_i$ and $v_j$ are not cut-vertices in different blocks, then $M_{i,j}=-e(U_i,U_j)=0$. Second, without loss of generality, if $v_i$ is not a cut-vertex but $v_j$ is a cut-vertex in a different block, then 
    \[M_{i,j}=\begin{cases}
        e(\{v_i\},U_j^c)& \text{if }v_i\in U_j\\
        -e(\{v_i\},U_j) & \text{otherwise}.
    \end{cases}\]
    In both cases $M_{i,j}=0$. Finally, if $v_i$ and $v_j$ are both cut vertices in different blocks,
    then 
    \[M_{i,j}=\begin{cases}
        e(U_i,U_j^c)& \text{if }v_i\in U_j\\
        -e(U_i,U_j) & \text{otherwise}.
    \end{cases}\]
    Thus, similarly to the sub-case above, $M_{i,j}=0$. Hence, we have the result.
\end{proof}

For $k$-trees we have the following.

\begin{corollary}
    Let $G$ be a $k$-tree and $G^k$ be its $k$-adjacency graph. 
    Then $$PL(G^k)P^t=\left(\bigoplus_{i=1}^{n-k} L(K_k)\right)\oplus [0]=\left(\bigoplus_{i=1}^{n-k} (k+1)I_k-J_k\right)\oplus [0].$$ 
\end{corollary}

We will finish this section extending Definition~\ref{def:P} to any connected graph, that is, graphs with possibly more than one unique path from any two vertices. 
For a graph $G$ with $n$ vertices, let $N$ be the $n\times n$ matrix with entry $n_{ij}= 1$ if vertex $i$ belongs to \emph{one} minimum path from vertex 1 to vertex  $j$; otherwise $n_{ij}=0$.
Furthermore, for a vertex $i$, let $\ell_{i}$ be the number of minimum paths from $1$ to $i$.
\begin{proposition}
    Let $N^{*}$ be the matrix with entry
    $$n_{ij}^{*}=\begin{cases}
        1,&\text{if } i=j\\
        -1,&\text{if } d(i,j)=1\text{ and } d(1,j)=1+d(1, i)\\
        \ell_{j}-1,&\text{if } \ell_{j}>2,~ d(1,i)=1\text{ and } d(1,j)=1+d(i,j)\\
        1,&\text{ if }\ell_{j}=2\text{ and } i=1\\
        0,&\text{otherwise}
    \end{cases}.$$
    Then $NN^{*}=I$ where $I$ is the identity matrix. 
\end{proposition}
\begin{proof}
    The $(i,j)$ entry of $NN^{*}$ is $\sum_{k=1}^{n}n_{ik}n_{kj}^{*}$ and hence the only non-zero contributions are those for which $n_{ik}\neq 0$ and $n_{kj}^{*}\neq 0$.
    Assume first $i>1$.
    For $k\neq i$, $n_{ik}\neq 0$ whenever $i$ belongs to a minimum path between $1$ and $k$, in consequence, $d(1,i)<d(1,k)$ and $d(1,k)\neq 1$, therefore $n_{ki}^{*}=0$ and the only contribution for the $(i,i)$ entry is $n_{ii}n_{ii}^{*}=1$.
    Let $j\neq i$. 
    If $i$ lies in a minimum path joining $1$ and $j$, then $n_{ij}n_{jj}^{*}=1$ and consider $k'$ an adjacent vertex to $j$ such that $d(1,j)=1+d(1,k')$ then $n_{ik}=1$ and $n_{ik'}n_{k'j}=-1$ which annihilates $n_{ij}n_{jj}^{*}$.

    If $j$ has at least $3$ minimum paths joining it with $1$ and $d(1,i)=1+d(i,j)$ then $i$ belongs to a minimum path from $1$ to $j$, thus $n_{ii}n_{ij}^{*}=\ell_{j}-1$.
    For each of those minimum paths of $j$ we may take a vertex $k'$ with $d(k',j)=1$ and $d(1,j)=1+d(1,k')$, then $n_{ik}=1$ and $n_{k'j}^{*}=-1$.
    Since $i$ lies between vertex $j$ and $1$, $n_{ij}n_{jj}^{*}=1$. 

    For $j>1$, if $k$ is such that $n_{kj}^{*}=-1$ then is annihilated by $n_{jj}^{*}=1$.
    If $j$ has exactly 2 minimum paths to vertex $1$, then $n_{1j}=1$ and there are two adjacent vertices $k',k''$ to $j$ with $n_{k'j}^{*}=n_{k''j}^{*}=-1$ and together with $n_{jj}^{*}=1$ and finally the entry $(1,j)$ of $NN^{*}$ is 0. 
\end{proof}
\begin{corollary}
    The matrix $N$ is unimodular. That is, $\det(N)=\pm1$.
\end{corollary}





\section*{Acknowledgement}
The research of C. A. Alfaro is partially supported by Sistema Nacional de Investigadoras e Investigadores grant number 220797.
The research of Jesús Uriel Medrano is partially supported by Secretaría de Ciencia, Humanidades, Tecnología e Innovación grant number 1343582.
The research of Iván Téllez Téllez is partially supported by Sistema Nacional de Investigadoras e Investigadores grant number 362694.

\bibliographystyle{plain}
\bibliography{bibliography}

@article{AVi,
  title={The structure of sandpile groups of outerplanar graphs},
  author={Alfaro, Carlos A and Villagr{\'a}n, Ralihe R},
  journal={Applied Mathematics and Computation},
  volume={395},
  pages={125861},
  year={2021},
  publisher={Elsevier}
}

@article {MR3083114,
    AUTHOR = {Liao, Yunhua and Fang, Aixiang and Hou, Yaoping},
     TITLE = {The {T}utte polynomial of an infinite family of outerplanar,
              small-world and self-similar graphs},
   JOURNAL = {Phys. A},
  FJOURNAL = {Physica A. Statistical Mechanics and its Applications},
    VOLUME = {392},
      YEAR = {2013},
    NUMBER = {19},
     PAGES = {4584--4593},
      ISSN = {0378-4371,1873-2119},
   MRCLASS = {05C31 (05C10)},
  MRNUMBER = {3083114},
       DOI = {10.1016/j.physa.2013.05.021},
       URL = {https://doi.org/10.1016/j.physa.2013.05.021},
}

@article {MR3046135,
    AUTHOR = {Comellas, Francesc and Miralles, Al\'icia and Liu, Hongxiao
              and Zhang, Zhongzhi},
     TITLE = {The number of spanning trees of an infinite family of
              outerplanar, small-world and self-similar graphs},
   JOURNAL = {Phys. A},
  FJOURNAL = {Physica A. Statistical Mechanics and its Applications},
    VOLUME = {392},
      YEAR = {2013},
    NUMBER = {12},
     PAGES = {2803--2806},
      ISSN = {0378-4371,1873-2119},
   MRCLASS = {05C82},
  MRNUMBER = {3046135},
       DOI = {10.1016/j.physa.2012.10.047},
       URL = {https://doi.org/10.1016/j.physa.2012.10.047},
}

@article {MR3479466,
    AUTHOR = {Krepkiy, I. A.},
     TITLE = {The sandpile groups of chain-cyclic graphs},
   JOURNAL = {Zap. Nauchn. Sem. S.-Peterburg. Otdel. Mat. Inst. Steklov.
              (POMI)},
  FJOURNAL = {Rossi\u iskaya Akademiya Nauk. Sankt-Peterburgskoe Otdelenie.
              Matematicheski\u i\ Institut im. V. A. Steklova. Zapiski
              Nauchnykh Seminarov (POMI)},
    VOLUME = {421},
      YEAR = {2014},
     PAGES = {94--112},
      ISSN = {0373-2703},
   MRCLASS = {05C25},
  MRNUMBER = {3479466},
       DOI = {10.1007/s10958-014-1961-5},
       URL = {https://doi.org/10.1007/s10958-014-1961-5},
}

@article {MR4023161,
    AUTHOR = {Chen, Haiyan and Mohar, Bojan},
     TITLE = {The sandpile group of a polygon flower},
   JOURNAL = {Discrete Appl. Math.},
  FJOURNAL = {Discrete Applied Mathematics. The Journal of Combinatorial
              Algorithms, Informatics and Computational Sciences},
    VOLUME = {270},
      YEAR = {2019},
     PAGES = {68--82},
      ISSN = {0166-218X,1872-6771},
   MRCLASS = {05C25 (05C10 05C20 05C30)},
  MRNUMBER = {4023161},
MRREVIEWER = {Timothy\ Y.\ Chow},
       DOI = {10.1016/j.dam.2019.07.020},
       URL = {https://doi.org/10.1016/j.dam.2019.07.020},
}

@article {MR3442497,
    AUTHOR = {Becker, Ryan and Glass, Darren B.},
     TITLE = {Cyclic critical groups of graphs},
   JOURNAL = {Australas. J. Combin.},
  FJOURNAL = {The Australasian Journal of Combinatorics},
    VOLUME = {64},
      YEAR = {2016},
     PAGES = {366--375},
      ISSN = {1034-4942,2202-3518},
   MRCLASS = {05C25 (05C50)},
  MRNUMBER = {3442497},
MRREVIEWER = {Norbert\ Seifter},
}

@article {MR1756151,
    AUTHOR = {Cori, Robert and Rossin, Dominique},
     TITLE = {On the sandpile group of dual graphs},
   JOURNAL = {European J. Combin.},
  FJOURNAL = {European Journal of Combinatorics},
    VOLUME = {21},
      YEAR = {2000},
    NUMBER = {4},
     PAGES = {447--459},
      ISSN = {0195-6698,1095-9971},
   MRCLASS = {05C25 (20K01)},
  MRNUMBER = {1756151},
       DOI = {10.1006/eujc.1999.0366},
       URL = {https://doi.org/10.1006/eujc.1999.0366},
}

@book {schrijver,
    AUTHOR = {Schrijver, Alexander},
     TITLE = {Theory of linear and integer programming},
    SERIES = {Wiley-Interscience Series in Discrete Mathematics},
      NOTE = {A Wiley-Interscience Publication},
 PUBLISHER = {John Wiley \& Sons, Ltd., Chichester},
      YEAR = {1986},
     PAGES = {xii+471},
      ISBN = {0-471-90854-1},
   MRCLASS = {90C05 (90C10)},
  MRNUMBER = {874114},
MRREVIEWER = {J\"urgen\ K\"ohler},
}

@book {cohen,
    AUTHOR = {Cohen, Henri},
     TITLE = {A course in computational algebraic number theory},
    SERIES = {Graduate Texts in Mathematics},
    VOLUME = {138},
 PUBLISHER = {Springer-Verlag, Berlin},
      YEAR = {1993},
     PAGES = {xii+534},
      ISBN = {3-540-55640-0},
   MRCLASS = {11Y40 (11Rxx 68Q40)},
  MRNUMBER = {1228206},
MRREVIEWER = {Joe\ P.\ Buhler},
       DOI = {10.1007/978-3-662-02945-9},
       URL = {https://doi.org/10.1007/978-3-662-02945-9},
}

@book {MR957919,
    AUTHOR = {Rotman, Joseph J.},
     TITLE = {An introduction to algebraic topology},
    SERIES = {Graduate Texts in Mathematics},
    VOLUME = {119},
 PUBLISHER = {Springer-Verlag, New York},
      YEAR = {1988},
     PAGES = {xiv+433},
      ISBN = {0-387-96678-1},
   MRCLASS = {55-01},
  MRNUMBER = {957919},
MRREVIEWER = {P.\ J.\ Kahn},
       DOI = {10.1007/978-1-4612-4576-6},
       URL = {https://doi.org/10.1007/978-1-4612-4576-6},
}

@book {MR256911,
    AUTHOR = {Harary, Frank},
     TITLE = {Graph theory},
 PUBLISHER = {Addison-Wesley Publishing Co., Reading, Mass.-Menlo Park,
              Calif.-London},
      YEAR = {1969},
     PAGES = {ix+274},
   MRCLASS = {05.40},
  MRNUMBER = {256911},
MRREVIEWER = {M.\ E.\ Watkins},
}

@article {MR5026280,
    AUTHOR = {Alfaro, Carlos A. and Medrano, Jes\'us Uriel and T\'ellez
              T\'ellez, Iv\'an},
     TITLE = {The {S}mith normal form of distance matrices of high
              dimensional trees},
   JOURNAL = {Comput. Appl. Math.},
  FJOURNAL = {Computational \& Applied Mathematics},
    VOLUME = {45},
      YEAR = {2026},
    NUMBER = {6},
     PAGES = {Paper No. 226, 12},
      ISSN = {2238-3603,1807-0302},
   MRCLASS = {05C50 (05E45 57M15 57Q05)},
  MRNUMBER = {5026280},
       DOI = {10.1007/s40314-025-03618-9},
       URL = {https://doi.org/10.1007/s40314-025-03618-9},
}

@article{aa,
  title={Enumeration of cospectral and coinvariant graphs},
  author={Abiad, Aida and Alfaro, Carlos A.},
  journal={Applied Mathematics and Computation},
  volume={408},
  pages={126348},
  year={2021},
  publisher={Elsevier}
}

@article {MR4131346,
    AUTHOR = {Lim, Lek-Heng},
     TITLE = {Hodge {L}aplacians on graphs},
   JOURNAL = {SIAM Rev.},
  FJOURNAL = {SIAM Review},
    VOLUME = {62},
      YEAR = {2020},
    NUMBER = {3},
     PAGES = {685--715},
      ISSN = {1095-7200,0036-1445},
   MRCLASS = {58A14 (05C50 20G10)},
  MRNUMBER = {4131346},
       DOI = {10.1137/18M1223101},
       URL = {https://doi.org/10.1137/18M1223101},
}

@article {MR4783080,
    AUTHOR = {Ribando-Gros, Emily and Wang, Rui and Chen, Jiahui and Tong,
              Yiying and Wei, Guo-Wei},
     TITLE = {Combinatorial and {H}odge {L}aplacians: similarities and
              differences},
   JOURNAL = {SIAM Rev.},
  FJOURNAL = {SIAM Review},
    VOLUME = {66},
      YEAR = {2024},
    NUMBER = {3},
     PAGES = {575--601},
      ISSN = {1095-7200,0036-1445},
   MRCLASS = {05C50 (58A14)},
  MRNUMBER = {4783080},
MRREVIEWER = {J\'ozef\ Dodziuk},
       DOI = {10.1137/22M1482299},
       URL = {https://doi.org/10.1137/22M1482299},
}

@article {MR4925290,
    AUTHOR = {Fan, Yi-Zheng and Wu, Hui-Feng and Wang, Yi},
     TITLE = {The largest {L}aplacian eigenvalue and the balancedness of
              simplicial complexes},
   JOURNAL = {J. Algebraic Combin.},
  FJOURNAL = {Journal of Algebraic Combinatorics. An International Journal},
    VOLUME = {61},
      YEAR = {2025},
    NUMBER = {4},
     PAGES = {Paper No. 53, 21},
      ISSN = {0925-9899,1572-9192},
   MRCLASS = {05E45 (05C65 47J10 55U05)},
  MRNUMBER = {4925290},
MRREVIEWER = {Yilun\ Shang},
       DOI = {10.1007/s10801-025-01419-1},
       URL = {https://doi.org/10.1007/s10801-025-01419-1},
}

@article {MR3077874,
    AUTHOR = {Horak, Danijela and Jost, J\"urgen},
     TITLE = {Spectra of combinatorial {L}aplace operators on simplicial
              complexes},
   JOURNAL = {Adv. Math.},
  FJOURNAL = {Advances in Mathematics},
    VOLUME = {244},
      YEAR = {2013},
     PAGES = {303--336},
      ISSN = {0001-8708,1090-2082},
   MRCLASS = {55U10 (18G30 31C20)},
  MRNUMBER = {3077874},
MRREVIEWER = {Paul\ G.\ Goerss},
       DOI = {10.1016/j.aim.2013.05.007},
       URL = {https://doi.org/10.1016/j.aim.2013.05.007},
}

@article {MR5088845,
    AUTHOR = {Chen, Xiaodan and Kou, Yongfang},
     TITLE = {Bounding the largest eigenvalue of signless {L}aplace operator
              on simplicial complexes},
   JOURNAL = {Graphs Combin.},
  FJOURNAL = {Graphs and Combinatorics},
    VOLUME = {42},
      YEAR = {2026},
    NUMBER = {4},
     PAGES = {Paper No. 57},
      ISSN = {0911-0119,1435-5914},
   MRCLASS = {05E45},
  MRNUMBER = {5088845},
       DOI = {10.1007/s00373-026-03055-3},
       URL = {https://doi.org/10.1007/s00373-026-03055-3},
}

@article {MR4061004,
    AUTHOR = {Pinheiro, Luc\'elia K. and Souza, Bruna S. and Trevisan,
              Vilmar},
     TITLE = {Determining graphs by the complementary spectrum},
   JOURNAL = {Discuss. Math. Graph Theory},
  FJOURNAL = {Discussiones Mathematicae. Graph Theory},
    VOLUME = {40},
      YEAR = {2020},
    NUMBER = {2},
     PAGES = {607--620},
      ISSN = {1234-3099,2083-5892},
   MRCLASS = {05C50},
  MRNUMBER = {4061004},
MRREVIEWER = {John\ T.\ Saccoman},
       DOI = {10.7151/dmgt.2280},
       URL = {https://doi.org/10.7151/dmgt.2280},
}

@misc{fallat2024minimumnumberdistincteigenvalues,
      title={Minimum number of distinct eigenvalues of distance-regular and signed Johnson graphs}, 
      author={Shaun Fallat and Himanshu Gupta and Allen Herman and Johnna Parenteau},
      year={2024},
      eprint={2411.00250},
      archivePrefix={arXiv},
      primaryClass={math.CO},
      url={https://arxiv.org/abs/2411.00250}, 
}

@article{ah,
  title={Cospectrality of graphs with respect to distance matrices},
  author={Aouchiche, Mustapha and Hansen, Pierre},
  journal={Applied Mathematics and Computation},
  volume={325},
  pages={309--321},
  year={2018},
  publisher={Elsevier}
}

@article{az,
    author = {Alfaro, Carlos A. and Zapata, Octavio},
    title = {The degree-distance and transmission-adjacency matrices},
    journal = {Computational and Applied Mathematics},
    volume = {43},
    year = {2024},
    pages = {351}
}

@article{bs,
  title={Cospectral graphs on 12 vertices},
  author={Brouwer, Andries E. and Spence, Edward},
  journal={Electronic Journal of Combinatorics},
  volume={16},
  number={1},
  pages={N20},
  year={2009}
}

@inproceedings{gm,
  title={Some computational results on the spectra of graphs},
  author={Godsil, Chris and McKay, Brendan},
  booktitle={Combinatorial Mathematics IV: Proceedings of the Fourth Australian Conference Held at the University of Adelaide August 27--29, 1975},
  pages={73--92},
  year={1976},
  organization={Springer}
}

@article{HS2004,
  title={Enumeration of cospectral graphs},
  author={Haemers, Willem H. and Spence, Edward},
  journal={European Journal of Combinatorics},
  volume={25},
  number={2},
  pages={199--211},
  year={2004},
  publisher={Elsevier}
}

@article{HW,
  title={Distance unimodular equivalence of graphs},
  author={Hou, Yaoping and Woo, Chingwah},
  journal={Linear and Multilinear Algebra},
  volume={56},
  number={6},
  pages={611--626},
  year={2008},
  publisher={Taylor \& Francis}
}

@article{kannan,
  title={Polynomial algorithms for computing the Smith and Hermite normal forms of an integer matrix},
  author={Kannan, Ravindran and Bachem, Achim},
  journal={SIAM Journal on Computing},
  volume={8},
  number={4},
  pages={499--507},
  year={1979},
  publisher={SIAM}
}

@book{Klivans,
  title={The Mathematics of Chip-firing},
  author={Klivans, Caroline J.},
  year={2018},
  publisher={CRC Press, Taylor \& Francis Group}
}

@article{stanley,
  title={Smith normal form in combinatorics},
  author={Stanley, Richard P.},
  journal={Journal of Combinatorial Theory, Series A},
  volume={144},
  pages={476--495},
  year={2016},
  publisher={Elsevier}
}

@misc{sage,
    author={The developers}, 
    title={{S}age {T}utorial ({R}elease 9.5)},
    year={2022},
    note={Located at: http://www.sagemath.org}
}

@article{vince,
  title={Elementary divisors of graphs and matroids},
  author={Vince, Andrew},
  journal={European Journal of Combinatorics},
  volume={12},
  number={5},
  pages={445--453},
  year={1991},
  publisher={Elsevier}
}

@article {GP,
    AUTHOR = {Graham, R. L. and Pollak, H. O.},
     TITLE = {On the addressing problem for loop switching},
   JOURNAL = {Bell System Tech. J.},
  FJOURNAL = {The Bell System Technical Journal},
    VOLUME = {50},
      YEAR = {1971},
     PAGES = {2495--2519},
      ISSN = {0005-8580},
   MRCLASS = {94.30},
  MRNUMBER = {289210},
MRREVIEWER = {M.\ A.\ Harrison},
       DOI = {10.1002/j.1538-7305.1971.tb02618.x},
       URL = {https://doi.org/10.1002/j.1538-7305.1971.tb02618.x},
}

@article {GL,
    AUTHOR = {Graham, R. L. and Lov\'asz, L.},
     TITLE = {Distance matrix polynomials of trees},
   JOURNAL = {Adv. in Math.},
  FJOURNAL = {Advances in Mathematics},
    VOLUME = {29},
      YEAR = {1978},
    NUMBER = {1},
     PAGES = {60--88},
      ISSN = {0001-8708},
   MRCLASS = {05C05},
  MRNUMBER = {480119},
MRREVIEWER = {R.\ A.\ Melter},
       DOI = {10.1016/0001-8708(78)90005-1},
       URL = {https://doi.org/10.1016/0001-8708(78)90005-1},
}

@book {MR755006,
    AUTHOR = {Munkres, James R.},
     TITLE = {Elements of algebraic topology},
 PUBLISHER = {Addison-Wesley Publishing Company, Menlo Park, CA},
      YEAR = {1984},
     PAGES = {ix+454},
      ISBN = {0-201-04586-9},
   MRCLASS = {55-01},
  MRNUMBER = {755006},
MRREVIEWER = {Christopher\ W.\ Stark},
}

@book {MR357214,
    AUTHOR = {Harary, Frank and Palmer, Edgar M.},
     TITLE = {Graphical enumeration},
 PUBLISHER = {Academic Press, New York-London},
      YEAR = {1973},
     PAGES = {xiv+271},
   MRCLASS = {05C30},
  MRNUMBER = {357214},
MRREVIEWER = {R.\ W.\ Robinson},
}

@article {MR228355,
    AUTHOR = {Harary, Frank and Palmer, Edgar M.},
     TITLE = {On acyclic simplicial complexes},
   JOURNAL = {Mathematika},
  FJOURNAL = {Mathematika. A Journal of Pure and Applied Mathematics},
    VOLUME = {15},
      YEAR = {1968},
     PAGES = {115--122},
      ISSN = {0025-5793},
   MRCLASS = {05.10},
  MRNUMBER = {228355},
MRREVIEWER = {M.\ D.\ Plummer},
       DOI = {10.1112/S002557930000245X},
       URL = {https://doi.org/10.1112/S002557930000245X},
}

@article {MR3007180,
    AUTHOR = {Gainer-Dewar, Andrew},
     TITLE = {{$\Gamma$}-species and the enumeration of {$k$}-trees},
   JOURNAL = {Electron. J. Combin.},
  FJOURNAL = {Electronic Journal of Combinatorics},
    VOLUME = {19},
      YEAR = {2012},
    NUMBER = {4},
     PAGES = {Paper 45, 33},
      ISSN = {1077-8926},
   MRCLASS = {05C30 (05E18)},
  MRNUMBER = {3007180},
       DOI = {10.37236/2615},
       URL = {https://doi.org/10.37236/2615},
}

@article {MR3213312,
    AUTHOR = {Gainer-Dewar, Andrew and Gessel, Ira M.},
     TITLE = {Counting unlabeled {$k$}-trees},
   JOURNAL = {J. Combin. Theory Ser. A},
  FJOURNAL = {Journal of Combinatorial Theory. Series A},
    VOLUME = {126},
      YEAR = {2014},
     PAGES = {177--193},
      ISSN = {0097-3165,1096-0899},
   MRCLASS = {05C30 (05C05 05C15)},
  MRNUMBER = {3213312},
MRREVIEWER = {Edward\ A.\ Bender},
       DOI = {10.1016/j.jcta.2014.05.002},
       URL = {https://doi.org/10.1016/j.jcta.2014.05.002},
}

@article {MR427152,
    AUTHOR = {Dewdney, A. K. and Harary, Frank},
     TITLE = {The adjacency graphs of a complex},
   JOURNAL = {Czechoslovak Math. J.},
  FJOURNAL = {Czechoslovak Mathematical Journal},
    VOLUME = {26(101)},
      YEAR = {1976},
    NUMBER = {1},
     PAGES = {137--144},
      ISSN = {0011-4642},
   MRCLASS = {05C99},
  MRNUMBER = {427152},
MRREVIEWER = {Richard\ A.\ Duke},
}

@article {MR4742211,
    AUTHOR = {Arseneva, Elena and Kleist, Linda and Klemz, Boris and
              L\"offler, Maarten and Schulz, Andr\'e{} and Vogtenhuber,
              Birgit and Wolff, Alexander},
     TITLE = {Adjacency graphs of polyhedral surfaces},
   JOURNAL = {Discrete Comput. Geom.},
  FJOURNAL = {Discrete \& Computational Geometry. An International Journal
              of Mathematics and Computer Science},
    VOLUME = {71},
      YEAR = {2024},
    NUMBER = {4},
     PAGES = {1429--1455},
      ISSN = {0179-5376,1432-0444},
   MRCLASS = {05C10 (05C42 05C62 68R10)},
  MRNUMBER = {4742211},
MRREVIEWER = {Joanna\ A.\ Ellis-Monaghan},
       DOI = {10.1007/s00454-023-00537-6},
       URL = {https://doi.org/10.1007/s00454-023-00537-6},
}

@misc{oeis,
    Author = {{OEIS Foundation Inc.}},
    Note = {Published electronically at \url{http://oeis.org}},
    Title = {The {O}n-{L}ine {E}ncyclopedia of {I}nteger {S}equences},
    Year = 2025}

\appendix

\section{Code for computing the $k$-{\it th} Laplacian matrix and the $k$-{\it th} auxiliar graph of a simplicial complex}\label{app:lapmatauxgra}

\begin{lstlisting}
def d_laplacian_matrix(S,l):
    V = S.n_faces(l-1)
    V.sort()
    print("Vertices of the " + str(l) + "-th adjacency graph")
    print(V)
    E = S.n_faces(l)
    nV = len(V)
    nE = len(E)
    A = Matrix(nV,nV)
    for i in range(nV):
        for j in range(i+1,nV):
            flag = False
            for e in E:
                if set(V[i]).issubset(e) and set(V[j]).issubset(e):
                    flag = True
                    break
            if flag:
                A[i,j] = 1
                A[j,i] = 1
    H = Graph(A, format="adjacency_matrix")
    print("The " + str(l) + "-th adjacency graph")
    H.show()
    L = H.laplacian_matrix()
    print("The " + str(l) + "-th Laplacian matrix")
    print(L)
    print("The SNF of the " + str(l) + "-th Laplacian matrix")
    print(L.elementary_divisors())
\end{lstlisting}

\section{Code to compute the simplicial annulus $O_n$}\label{sec:annulussage}
\begin{lstlisting}
def simplicial_annulus(n):
    C1 = graphs.CycleGraph(n)
    C2 = graphs.CycleGraph(2*n)
    G = C1.disjoint_union(C2)
    for i in range(n):
        if 2*i-1 < 0:
            for j in [2*n-1,2*i,2*i+1]:
                G.add_edge(((0,i),(1,j)))
        else:
            for j in [2*i-1,2*i,2*i+1]:
                G.add_edge(((0,i),(1,j)))
    d = {v:n*v[0]+v[1] for v in G.vertices()}
    G.relabel(d)
    L = list(sage.graphs.cliquer.all_cliques(G, 3, 3))
    if n == 3:
        L.remove([0,1,2])
    return SimplicialComplex(L)
\end{lstlisting}

\section{Code to generate all $k$-trees with $n$ vertices}\label{sec:generation of k trees}
\begin{lstlisting}
def k_trees(k,nmax):
    G = [[graphs.CompleteGraph(k)]]
    for n in range(nmax-k):
        L = []
        for g in G[n]:
            for clique in sage.graphs.cliquer.all_cliques(g,k,k):
                h = g.copy()
                h.add_vertex()
                h.add_edges([(g.order(),v) for v in clique])
                bandera = True
                for l in L:
                    if h.is_isomorphic(l):
                        bandera = False
                        break
                if bandera:
                    L.append(h)
        G.append(L)
    return G
\end{lstlisting}

\section{Code to compute the SNF of the $d$-Laplacian matrix of a $k$-tree}\label{sec:k Laplacian matrix}
\begin{lstlisting}
def d_laplacian_matrix(g,l):
    V = list(sage.graphs.cliquer.all_cliques(g,l,l))
    E = list(sage.graphs.cliquer.all_cliques(g,l+1,l+1))
    nV = len(V)
    nE = len(E)
    A = Matrix(nV,nV)
    for i in range(nV):
        for j in range(i+1,nV):
            flag = False
            for e in E:
                if set(V[i]).issubset(e) and set(V[j]).issubset(e):
                    flag = True
                    break
            if flag:
                A[i,j] = 1
                A[j,i] = 1
    H = Graph(A, format="adjacency_matrix")
    L = H.laplacian_matrix()
    print("d = " + str(l) + ", SNF of L^d(T): " + str(L.elementary_divisors()))
\end{lstlisting}

\end{document}